\documentclass[11pt]{article}
\usepackage[utf8]{inputenc}

\usepackage{amsmath} 
\usepackage{amsthm,amssymb,mathrsfs} 
\usepackage[T1]{fontenc}
\usepackage[utf8]{inputenc}
\usepackage{ae,aecompl}
\usepackage{times}
\usepackage{cite}

\usepackage[colorlinks=true,
            linkcolor=blue,
            urlcolor=cyan,
            citecolor=green]{hyperref}

\usepackage{verbatim}
\usepackage{tikz}
\usepackage[numeric,backrefs]{amsrefs}
\usepackage{enumerate}
\usepackage[cmtip, all]{xy}
\usepackage{booktabs}
\usepackage{pb-diagram}
\usepackage{bbm}
\usepackage[top=2.54cm,bottom=2.54cm,left=2.54cm,right=2.54cm]{geometry}
\usepackage{tikz-cd} 
\numberwithin{equation}{section} 
\usepackage{bbding}

\newcommand{\rA}{{\rm A}}

\newcommand{\rF}{{\rm F}}
\newcommand{\rG}{{\rm G}}
\newcommand{\rH}{{\rm H}}
\newcommand{\rI}{{\rm I}}

\newcommand{\rL}{{\rm L}}

\newcommand{\rN}{{\rm N}}

\newcommand{\rR}{{\rm R}}

\newcommand{\rT}{{\rm T}}

\newcommand{\rV}{{\rm V}}

\newcommand{\bH}{{\bf H}}

\newcommand{\cA}{\mathcal{A}}

\newcommand{\cC}{\mathcal{C}}

\newcommand{\cF}{\mathcal{F}}

\newcommand{\cL}{\mathcal{L}}
\newcommand{\cM}{\mathcal{M}}

\newcommand{\cR}{\mathcal{R}}
\newcommand{\cS}{\mathcal{S}}

\newcommand{\cX}{\mathcal{X}}
\newcommand{\cY}{\mathcal{Y}}
 
\newcommand{\fg}{{\mathfrak g}}

\newcommand{\frgl}{{\mathfrak gl}}

\newcommand{\fm}{{\mathfrak m}}

\newcommand{\fX}{{\mathfrak X}}

\newcommand{\R}{\mathbb{R}}
\newcommand{\C}{\mathbb{C}}
\renewcommand{\H}{\bH}

\newcommand{\su}{\mathfrak{su}}
\newcommand{\so}{\mathfrak{so}}

\newcommand{\SO}{{\rm SO}}

\newcommand{\Spin}{{\rm Spin}}
\newcommand{\SU}{{\rm SU}}

\renewcommand{\det}{\mathop\mathrm{det}\nolimits}

\newcommand{\End}{{\mathrm{End}}}

\renewcommand{\epsilon}{\varepsilon}

\newcommand{\Hom}{{\mathrm{Hom}}}

\newcommand{\Ric}{{\rm Ric}}

\newcommand{\ad}{\mathrm{ad}}

\newcommand{\delbar}{\bar \del}
\newcommand{\del}{\partial}

\newcommand{\dvol}{\mathop\mathrm{dvol}\nolimits}

\newcommand{\ind}{\mathop{\mathrm{index}}}
\newcommand{\indT}{\mathop{\mathrm{index}_T}}

\renewcommand{\Re}{\mathop{\mathrm{Re}}}

\newcommand{\tr}{\mathop{\mathrm{tr}}\nolimits}

\newcommand{\vol}{\mathrm{vol}}

\newcommand{\zbar}{\overline z}
\newcommand{\ol}[1]{\overline{#1}}
\newcommand{\wt}[1]{\widetilde{#1}}

\newcommand{\qandq}{\quad\text{and}\quad}
\newcommand{\qforq}{\quad\text{for}\quad}

\newcommand{\qwhereq}{\quad\text{where}\quad}  
\newcommand{\whereq}{\text{where}\quad}

\def\<{\mathopen{}\left<}
\def\>{\right>\mathclose{}}
\def\({\mathopen{}\left(}
\def\){\right)\mathclose{}}

\usepackage{multicol, color}

\definecolor{gold}{rgb}{0.85,.66,0}
\definecolor{cherry}{rgb}{0.9,.1,.2}
\definecolor{burgundy}{rgb}{0.8,.2,.2}
\definecolor{orangered}{rgb}{0.85,.3,0}
\definecolor{orange}{rgb}{0.85,.4,0}
\definecolor{olive}{rgb}{.45,.4,0}
\definecolor{lime}{rgb}{.6,.9,0}
\definecolor{green}{rgb}{.2,.7,0}
\definecolor{grey}{rgb}{.4,.4,.2}
\definecolor{brown}{rgb}{.4,.3,.1}

\newtheorem{theorem}{Theorem}[section]
\theoremstyle{Theorem}

\newtheorem{proposition}{Proposition}[section] 
\theoremstyle{Proposition}

\newtheorem{lemma}[proposition]{Lemma}  

\theoremstyle{remark} 
\newtheorem{remark}[proposition]{Remark} 

\theoremstyle{definition}
\newtheorem{definition}[proposition]{Definition}  
\newtheorem{example}[proposition]{Example}

\newtheorem{notation}[proposition]{Notation}

\hypersetup{linkbordercolor=blue}
\DeclareMathOperator{\Gl}{Gl}
\DeclareMathOperator{\Ker}{ker}
\DeclareMathOperator{\Ima}{Im} 

\DeclareMathOperator{\Tr}{Tr}

\DeclareMathOperator{\Span}{ {\rm Span}}

\newcommand{\dtzero}{\left.\frac d{dt}\right|_{t=0}}

\newcommand{\dtzer}{\left.\frac {d^2}{dt^2}\right|_{t=0}}

\newcommand{\ii}{{ \textbf{i}}}  

\usepackage{tikz}
\usetikzlibrary{shapes.geometric}
\newcommand{\warningsign}{\tikz[baseline=-.75ex] \node[shape=regular polygon, regular polygon sides=3, inner sep=0pt, draw, thick] {\textbf{!}};}
\usepackage[textsize=footnotesize,textwidth=3.5cm]{todonotes}
\usepackage[affil-it]{authblk}

 \newcommand{\quo}[2]{{\raisebox{.2em}{$#1\!$}\left/\raisebox{-.2em}{$#2$}\right.}} 

\begin{document}

\title{A Weitzenböck formula on Sasakian holomorphic bundles}
\author[1]{L. Portilla} 
\author[2]{E. Loubeau}
\author[3]{H. N. Sá  Earp}

\affil[1,2]{Univ. Brest, CNRS UMR 6205, LMBA, F-29238 Brest, France}
\affil[3]{Universidade Estadual de Campinas (Unicamp), SP, Brazil}

\date{\today}
\maketitle
\begin{abstract} 
This work seeks to advance the understanding of the smooth structure of the moduli space $\cM$ of self-dual contact instantons (SDCI) on Sasakian $7$-manifolds. A neighborhood of  a smooth point of $\cM$ is locally modeled on the first cohomological group of an   elliptic complex \eqref{eq:complex L D}. In \cite{portilla2023} Portilla and Sá Earp constructed the moduli space $\cM^\ast$ of SDCI and  provide a cohomological obstruction to the smoothness  of $\cM^\ast$, in terms of the basic cohomology group $\rH_B^2$ of \eqref{eq:BasicComplexDeformation}. In this paper we study conditions under which this obstruction disappears, by computing a  Weitzenböck formula and using a Bochner-type method to obtain a vanishing theorem for $\rH_B^2$. Given an SDCI $\nabla\in \cA(E)$ on a  Sasakian bundle  $E\to M$, we find sufficient conditions for the vanishing of $\rH_B^2$ in the positivity of a couple of operators $\cR$ and $\cF$ depending on the curvatures $F_\nabla$ of $\nabla$ and the Riemann curvature of the Sasakian metric $g$. In particular, we find that if $M$ is transversely Ricci positive and $\cF$ positive, the moduli space of SDCI must be smooth. However, in general, the operator $\cF$ is not positive definite and we describe bundles over the Stiefel manifold $V^{5,2}$ for which it is the case. Finally, we show that when the energy of $F_A$ is less than the first non-zero eigenvalue of $\Ric^T$ then $\rH^2_B$ vanishes.  
\end{abstract}
\tableofcontents
\newpage
\section{Introduction}

This is a preliminary version on joint work in progress, pertaining to the postdoctoral research activities of the first-named author. We advise against citations directed at numbered items, since sectioning will likely change in subsequent versions. While the material will be revised, the third-named author is not responsible for the content of this version. 

\begin{notation}\label{not:1}
From now on this paper, unless otherwise stated, $M$ denotes a Sasakian manifold with Sasakian structure $\cS=(\eta,\xi,g,\Phi)$[cf. Definition \ref{thm:app sasakian manifold}]. $E\to M$ will denote a complex vector $ \rG$-bundle and $\cA(E)$ the space of connections on $E$. The adjoint bundle $\fg_E$ is the subbundle of  $\End(E)$ whose fibre at a point consists of all those endomorphisms  $T\in\End(E_x)$ whose matrix representation lies in $\fg$. We will suppose $\rG$ to be a compact Lie group, so that the Lie algebra $\fg$ of $\rG$ admits some Ad-invariant  inner product.  
\end{notation}
The study of higher dimensional instantons \cite{donaldson1998gauge,wang2012higher} is particularly interesting in dimension $7$ since we can also define $\mathrm{G}_2$-structures. It is expected that the study of solutions of the instanton equations will yield new invariants, just as the original Donaldson’s invariants in $4$-manifold \cite{Donaldson1990}, Casson invariant for flat connections over $3$–manifolds \cite{donaldson2002floer} and, particularly relevant to this work, the  \emph{Donaldson invariants} counting $\mathrm{G}_2$-instantons on $7$-manifolds with holonomy in 
$\mathrm{G}_2$ \cite{Donaldson2011} or $\Spin(7)$-instantons on $8$-manifolds with holonomy $\Spin(7)$ or $\SU(4)$ \cite{cao2014donaldson,Donaldson2011}.\\

 A connection  $\nabla$ on $E\to M$ is called $\lambda$-\emph{contact instanton} (CI) if its curvature $2$-form satisfies the following equation \cite{portilla2023}
\begin{equation}
\label{eq: instanton}
T_\eta(\rF_\nabla)=\ast(\eta\wedge d\eta\wedge\rF_\nabla) = \lambda \rF_\nabla, \qforq \lambda=-1,1 \text{ or } -2.
\end{equation}
A connection $\nabla$ solution of \eqref{eq: instanton} for $\lambda=1$ is called a \emph{selfdual contact instanton}(SDCI) and we denote by $\cM^\ast$ the moduli space of SDCI, i.e. solutions of \eqref{eq: instanton} modulo the group gauge action, cf. \cite{portilla2023} Under certain hypothesis, a Sasakian manifold can also be equipped with a $\mathrm{G}_2$-structure, e.g. contact Calabi-Yau manifolds \cite{habib2015some}, and studies of $\mathrm{G}_2$-instantons have appeared in the literature \cites{SaEarp2015,Calvo-Andrade2016}. When a Sasakian manifold $M$ is endowed with a $\mathrm{G}_2$-structure, one can to study three interrelated notions of instantons:  SDCI, $\mathrm{G}_2$-instantons  and transversely Hermitian Yang-Mills connections (tHYM). A description of the interaction between the different types of instantons can be found in \cite[Theorem~1.1]{portilla2023}, together with a \emph{cohomological obstruction} to the smoothness of $\cM^{\ast}$.  In this paper we pursue investigations on $\cM^{\ast}$ and seek conditions removing this cohomological obstruction, i.e. 
$\cM^{\ast}$ is a smooth finite dimensional manifold.

\subsection*{Review of the local description of the moduli space of contact instantons}
\label{sec:1}  

We summarize the main results on the construction of the moduli space $\cM^{\ast}$  of SDCI \eqref{eq: instanton}, the main reference for this Section is \cite{portilla2023}. We denote by $\Omega^k_H(M)$ the space of \textit{transversal} $k$-forms on $M$, i.e. $k$-forms on $M$ satisfying  $i_\xi\alpha=0$. The $k$-form $\alpha$ is called \textit{basic} if it is transversal and satisfies  $i_\xi d \alpha=0$ and we denote the space of basic $k$-forms by $\Omega^k_B(M)$. Then we have the  splitting
\eqref{eq:DecompostionKforms}
$$ 
\Omega^k(M)=\Omega^k_H(M)\oplus \eta\wedge \Omega^{k-1}_H(M).
$$ 
The operator
$  
\rT_\eta\colon \alpha\in  \Omega^2(M)\mapsto\ast(\eta\wedge d\eta\wedge\alpha)\in \Omega^2(M)   
$  
further splits the bundle of transverse  $2$-forms into eigenspaces associated to the eigenvalues $\{1,-1,-2\}$ as follows cf. \cite[eq.~1.6]{portilla2023} 
\begin{equation}
\label{eq:2formsDecomposition}
   \Omega^2(M) =\underbrace{\Omega^2_1(M) \oplus\Omega^2_6(M)\oplus \Omega^2_8(M)}_{\text{transversal forms}} \oplus \Omega^2_V(M).
   \end{equation}
The subscript of $\Omega^2_{j}$ denotes the  rank of $\Omega^2_{j}$ and the rank $6$ bundle $\Omega^2_V$, called the bundle of \textit{vertical forms}, can be locally interpreted as $2$-forms with $\eta$-component. The decomposition~\eqref{eq:2formsDecomposition} naturally extends to  $\fg_E$-valued  $2$-forms on $M$.

If $\nabla\in \cA(\fg_E)$ is an SDCI then its curvature $2$-form $\rF_\nabla$ is in $\Omega^2_8(\fg_E)$, while an \emph{anti-selfdual contact instanton }(ASDCI) is a solution of \eqref{eq: instanton} for $\lambda=-1$, so that 
$\rF_\nabla\in \Omega^2_6(\fg_E)$. The third component,  $\Omega^2_1(\fg_E)= \omega\otimes\Omega^0(\fg_E)$ is locally generated by the fundamental transverse Kähler form $\omega:=d\eta$. 
In order to describe the structure of $\cM^{\ast}$, we introduce the algebraic ideal $\rI:=\langle\Omega^2_8(\fg_E) \rangle$ generated by $\Omega^2_8(\fg_E)$ in the graded Lie algebra $\Omega^\bullet(\fg_E)$ and define the graded Lie algebra   
$ 
L^\bullet:=\quo{\Omega^\bullet(\fg_E)}{\rI}.  
$ 
The spaces $\rL^k$  are characterized in \cite[Proposition.~4.5]{portilla2023} as follows: 
\begin{equation}\label{eq:Lk spaces}
\begin{array}{lll}
\rL^0  \cong\Omega^0(\fg_E),&\rL^1  \cong\Omega^1(\fg_E),  & \rL^2 \cong\Omega^{2}_{6\oplus 1} (\fg_E)\oplus\eta\wedge\Omega^{1}_H(\fg_E),  \\[3pt] 
\rL^3 \cong\eta\wedge \Omega^{2}_{6\oplus 1} (\fg_E), & \rL^k=0\quad \text{for}\; k\geq 4.&
\end{array}
\end{equation}
where $\Omega^{2}_{6\oplus 1} (\fg_E):=\Omega^{2}_{6}(\fg_E) \oplus\Omega^{2}_{1} (\fg_E)$, and they fit in an elliptic deformation complex 
\cite[Proposition~4.7]{portilla2023} 
\begin{equation}
\label{eq:complex L D}
C^\bullet\colon\quad  0\xrightarrow[{}]{} \rL^0\xrightarrow[{}]{D_0} \rL^1\xrightarrow[{}]{D_1} \rL^2\xrightarrow[{}]{D_2} \rL^3\xrightarrow[{}]{} 0,
\end{equation}  
where $D_k$ denotes the quotient of the extension $d^\nabla\colon \Omega^k(\fg_E)\to \Omega^{k+1}(\fg_E)$  which descends to $L^k$. The complex \eqref{eq:complex L D} is  referred  to as the  \emph{deformation complex associated}  to  $\nabla$. We denote by $\rH^i(C)$, for $i=0,1,2,3$, the cohomological groups  of \eqref{eq:complex L D} and by  
$a_i= \dim(\rH^i(C))$ their respective dimensions.
The space of infinitesimal deformations of an SDCI is modeled by $\rH^1(C)$ and since the index of a elliptic complex on a  odd-dimensional manifold vanishes, we have
$$ 
a_0-a_1+a_2-a_3=0.
$$ 
If $\nabla $ is irreducible, one can  show that $a_0=0$,  therefore  the expected  dimension of $\cM^\ast$ is given by $  a_1=a_2-a_3.$ The problem at this point is that one cannot expect obstructed instantons in general (those for which $a_2=0$) because then, using a Gysin sequence (cf. \cite[Proposition~4.17]{portilla2023}) one can show that $a_3=0$, giving a $0$-dimensional moduli space. However, the obstruction can be obtained in the \emph{basic cohomology}. Note that $D_k$ in \eqref{eq:complex L D} restricts to  basic forms, hence we have the \emph{basic deformation complex }  
\begin{equation}
\label{eq:BasicComplexDeformation}
0\xrightarrow[{}]{} \rL^0_B\xrightarrow[{}]{D_B} \rL^1_B\xrightarrow[{}]{D_B}\rL^2_B  \xrightarrow[{}]{} 0,
\end{equation} 
Denote by $\rH_B^k$ the cohomological groups of \eqref{eq:BasicComplexDeformation} and $h^k=\dim H_B^k$ their dimensions, 
the transverse index of $\nabla $ is defined as the index of \eqref{eq:BasicComplexDeformation}, i.e.
\begin{equation}\label{eq:index T}
\indT(\nabla)= h^0- h^1+h^2.
\end{equation}
\begin{remark}\label{rem:obstruction}
Despite \eqref{eq:BasicComplexDeformation} not being elliptic (it is in fact transversely elliptic), one can prove that $\rH^1_B\cong \rH^1$ \cite[Corollary~4.18]{portilla2023}. Furthermore, the vanishing of $\rH^2_B$ is the topological obstruction to the smoothness of $\cM^\ast$ \cite[Corollary.~4.26~ \& Proposition.~4.27]{portilla2023} and, in this case, $\dim(T_{[A]}\cM^\ast)=-\ind_T(\nabla)$.    
\end{remark}
We now describe how transversal Hodge theory and the transversal Kähler foliation describe the moduli space of SDCI. 
The restriction of $\Phi\in \End(TM)$ to $\rH$ satisfies $(\Phi\vert_{ \rH_x})^{2}=-\rI$, hence one can define the type $(p,q)$ of a transverse $p+q$-forms \eqref{eq:Quebra1} as a section of 
\begin{equation}\label{eq:HcSplit}
\rH^{p,q}:= \Lambda^{p} (\rH^{1,0})^\ast\otimes \Lambda^{q} (\rH^{0,1})^\ast \subset\Lambda^{p+q} (\rH)^\ast\otimes_{\R}\C.   
\end{equation}
We naturally extend the bidegree splitting \eqref{eq:HcSplit} to $\fg_E$-valued $k$-forms by setting 
$\Omega^{p,q}(\fg_E):=\Gamma(\Lambda^{p,q}\rH^\ast_\C \otimes \fg_E),$ and consider the metric on $\Omega^k(\fg_E)$   induced by tensoring the metric in  $\Lambda^kTM$ and a metric in  $\fg_E$. We also  extend the usual conjugation of forms in $\Omega^k(TM_\C^\ast)$ to sections of $\Omega^{p,q}(\fg_E)$ by conjugating the form part, so that 
$\ol{\Omega^{p,q}(\fg_E)}\cong\Omega^{q,p}(\fg_E)$ (cf. Section \ref{sec:inner product} for details).
We have  the following characterizing result which describes the irreducible split of $2$-forms on $M$ \cite[Lemma 1.2]{portilla2023}:

\begin{lemma}
\label{lem:SDcharacterization}
Let $(M,\cS)$ be a Sasakian manifold, then
\begin{enumerate}[(i)]

\item  A $2$-form $\alpha$ is in $\Omega^2_8$ if and only if it is of type $(1,1)$ and orthogonal to $\omega$.
\item A $2$-form $\alpha$ is in $\Omega^2_6$  if and only if  $\alpha=\beta +\ol{\beta} $ for a $2$-form $\beta$ of type $(2,0)$.  
\end{enumerate} 
\end{lemma}
According to \eqref{eq:2formsDecomposition} and Lemma~\ref{lem:SDcharacterization}, we have the irreducible splitting:  
\begin{equation}\label{eq:decompositionPQ}
\Omega^{2,0}(M) \oplus \Omega^{0,2}(M) \cong\Omega^2_6, \quad \Omega^{0}(M)\otimes d\eta\cong\Omega^{2}_{ 1} \qandq  \Omega^{1,1}_{\perp}(M)\cong \Omega^2_8 
\end{equation}
In particular, the following result is a consequence of Lemma \ref{lem:SDcharacterization} \cite[Proposition~2.5]{portilla2023} 
\begin{proposition}\label{prop: selfdual = HYM}
Let $E\to M$ be a complex $\rG$-bundle with adjoint bundle $\fg_E$. A connection $\nabla\in \cA(E)$ is a SDCI if and only if $\rF_\nabla\in\Omega^{1,1}(\fg_E),$ $\hat{\rF}_\nabla:=\langle \rF_\nabla,\omega \rangle=0 $ and 
$\ol{\rF}_\nabla=\rF_\nabla.$  
\end{proposition}
Let $\nabla\in \cA$ be a connection, the bidegree \eqref{eq:HcSplit} splits the extended differential operator as $d_\nabla=\partial_\nabla +\ol{\partial}_\nabla$ \eqref{eq:partial_B} into $(1,0)$ and $(0,1)$ components, acting on $\Omega^{p,q}_{H}(\fg_E)$ according to bi-degree decomposition:
\[
\begin{diagram}
\node{}\node {\Omega^{p+1,q}_H(\fg_E )}\arrow{se,t}{\partial_\nabla^\ast }\arrow{e,t}{\partial_\nabla}
\node{\Omega^{p+2,q}_H(\fg_E )} \\ 
\node{\Omega^{p,q}_H(\fg_E )} \arrow{ne,t}{\partial_\nabla} \arrow{se,r}{\bar{\partial}_\nabla} 
\node[2]{\Omega^{p,q}_H(\fg_E )} \\ 
\node{}\node {\Omega^{p,q+1}_H(\fg_E )}\arrow{ne,b}{\bar{\partial}_\nabla^\ast}\arrow{e,b}{\bar{\partial}_\nabla}
\node{\Omega^{p,q+2}_H(\fg_E )}  
\end{diagram} 
\]  
We define the operator
$ 
d_7:=P\circ d_\nabla\colon\Omega^1(\fg_E)\to\Omega_{6\oplus 1}^2(\fg_E) 
$ where  $P\colon \Omega^2(\fg_E)\to \Omega_{6\oplus 1}^2(\fg_E )$ is the natural projection. Let   $\alpha =\alpha^{1,0}+\alpha^{0,1}$ be a $1$-form, where $ \alpha^{0,1}\in \Omega^{0,1}(\fg_E ) $  and $\alpha^{1,0} \in \Omega^{1,0}(\fg_E )$, we have
 \begin{equation}
 \label{eq:def d7}
 d_7\alpha= 
 \partial_\nabla\alpha^{1,0}+ \ol{\partial}_\nabla\alpha^{0,1}+ \langle \partial_\nabla\alpha^{0,1}+\ol{\partial}_\nabla\alpha^{1,0},\omega\rangle\otimes\omega.
 \end{equation}
We know \cite[Proposition ~3.4]{portilla2023} that if $E\to M$ is a complex $\rG$-bundle on a compact, connected Sasakian $7$-manifold $(M, \cS)$ and $\fg_E$ is its adjoint bundle, then each SDCI $\nabla$ induces a deformation complex
\begin{equation}
\label{eq:Complexd7}
B\bullet:\quad 0\to\Omega^0_H(\fg_E) \xrightarrow[{  }]{ d_\nabla}\Omega^1_H(\fg_E) \xrightarrow[{  }]{d_7}\Omega_{6\oplus 1}^2(\fg_E)\to 0
\end{equation}
and  the tangent space   $T_{[\nabla]}\cM^\ast$  at $[\nabla]$  is isomorphic to  $ \rH^1_B:=\frac{\Ker(d_7)}{\Ima(d_\nabla)}$.\\

As observed in Remark \ref{rem:obstruction}, the topological obstruction for smoothness at $[\nabla]\in \cM^\ast$ is $\rH^2_B=0$. We use a Bochner-type argument to find conditions under which it vanishes. We compute a Weitzenböck formula (analogous to the one in \cite[Theorem~3.10]{bourguignon1981stability}) involving both the curvature of the base manifold and of $\nabla\in \cA(E)$. Under suitable hypotheses, we obtain a vanishing theorem and its consequences, as summarized in the following section.

\subsection*{Description of our main results} 
\label{sec:results} 
This section roughly summarizes the results of this work. The cornerstone of the article is a Weitzenböck formula obtained in 
Proposition \ref{prop:bochner}, relating the action of the transverse Dolbeau Laplacian acting on $(2,0)$-forms with certain  algebraic operations $\cF$ and $\cR$ as follows:  for   
$\varphi = \sum_{\gamma,\tau}\varphi_{\gamma\tau}dz^\gamma\wedge dz^\tau \in \Omega^{2,0}(\fg_E )$, we have the following  formula
\begin{equation*}  
\left( \Delta_{\partial_\nabla}\varphi \right)_{\mu\nu}
=-\sum_{\alpha\beta} g^{\alpha\bar{\beta}}\widetilde{\nabla}_{\bar{\beta}}\widetilde{\nabla}_{\alpha}\varphi_{\mu\nu} 
  - {(\cF\varphi)}_{\mu\nu} - {(\cR\varphi)}_{\mu\nu},  
\end{equation*}
here    
$   
{(\cF\varphi)}_{\mu\nu}   
$ 
and  
$ 
{(\cR\varphi)}_{\mu\nu}   
$  
define endomorphisms  $\cF$  and $\cR$ on $\Omega^{2,0}(\fg_E)$ (cf. Section~\ref{sec:weitzenbock}). They depend on the curvature $F_\nabla$ and the transverse Ricci curvature respectively (cf. Definition \ref{def:RicciT}).\\

A consequence of this Weitzenböck formula, is the vanishing theorem of the obstruction $\H_B^2(\fg_B)$, proved in Theorem \ref{thm:main theorem}. 
With aid of the vanishing theorem \ref{thm:main theorem}, we can show that, in fact, in weaker hypothesis  [cf. Remark \ref{rem:varphiMaintheorem} and Proposition \ref{prop:Pararell}]    all element in $H^2_B$ \eqref{eq:Complexd7} vanish or is covariantly constant, i.e.  $\nabla\varphi=0$  for all $\varphi\in H^2_B$.\\

Finally in Proposition \ref{prop:Ricci}, we can provide a sufficient condition to achieve the hypotheses in the Vanishing Theorem, namely  $\int_M\langle\cR\varphi,\varphi\rangle>0$ for all $0\neq \varphi\in \Omega^{2,0}$ and we establish that this depends on the transverse Ricci curvature. This allows to conclude that 
under hypotheses of  Proposition \ref{prop:Ricci}   the moduli space $\cM^\ast$ of irreducible SDCI carries a natural smooth structure and the Sasakian structure of $M$ induces a natural Kähler metric on $\cM^\ast$. To completeness the previous result we can prove that if $\Vert F_\nabla\Vert $ is small enough (cf. Proposition \ref{prop:energyF})  then $\int_M\{\langle\cF\varphi,\varphi\rangle+\langle\cR\varphi,\varphi\rangle \}> 0 $, and consequently $\rH^2_B=0$. Furthermore in Proposition \ref{cor:stabilityF} we provide condition under which $\nabla$ is stable as a Yang-Mills connection. 

\vspace{1cm}

\noindent\textbf{Outline:} 
Section \ref{sec:preliminares} fixes notation and recalls well-known facts. Section \ref{sec:weitzenbock} computes the Weitzenböck formula, the pillar of the article, which we use, in Section~\ref{sec:vanishing}, to obtain the vanishing theorem in its most general form. Section \ref{sec:positivity} revisits the Weitzenböck formula to obtain a Bourguignon-Lawson type of result. The appendix contains relevant properties on transversely Kähler geometry.\\

\noindent\textbf{Acknowledgments and funding:} L.P. was funded  by Centre Henri Lebesgue (CHL) under the grant \textbf{ANR-11-LABX-0020-01}.  He also would like to thanks the LMBA  at Université de Bretagne Occidentale for hospitality; 
All three authors are members of the Fapesp-ANR BRIDGES [2021/04065-6] collaboration.

\section{Preliminaries} 
\label{sec:preliminares} 

Topics in this section are  well known, they are mainly included to fix the notation.

\subsection{Connections and curvature}
\label{sec: connection and curvature locally}
Let $\nabla$ be a connection on the bundle $E$ of rank $r$. In a local frame $\{s_i\}\subset \Gamma(E)$, $\nabla$ defines a matrix-valued $1$-form called the \emph{connection form} relative to $\{s_i\}$ by 
 $ \nabla s_i= \sum_j\rA^j_i s_j$,  
 \begin{equation}\label{eq:connection matrix}
 \rA^j_i = \sum_k\rA^{\;\;j}_{k i} dx^k .    
 \end{equation}
The space of connections $\cA(E)$ is an affine space modeled on $\Omega^1(\End(E))$, 
so elements in $\Omega^1(\End(E))$ are referred to as connections. A representation of the group $\rG$ in $\Gl(n,\C)$ allows for an identification of $\fg_E$ as a subbundle of $\End(E)$ and  $\fg_E$ becomes the real subbundle whose fiber are endomorphisms $T_x \colon E_x \to E_x$ such with local representation in $\fg\subset \frgl(r,\C)$.\\

An element $\nabla\in \cA(E)$ is called a \emph{$\rG$-connection} if its  connection form $A$ lies in $\Omega^1(\fg_E)$ for any local  trivialization. We will denote the space of  $\rG$-connections as
$$
\cA(E)=\{ \nabla+A\;\vert A\in \Omega^1(\fg_E)\}. 
$$
As common in Yang-Mills theory, taking $\rG=U(r)$ or $O(2r)$ means that $\rG$-connections compatible with the metric: 
$$
d\langle s_1,s_2\rangle=\langle\nabla s_1,s_2\rangle+\langle s_1,\nabla s_2\rangle, \qwhereq s_1,s_2\in \Gamma(E),
$$
and connections will systematically be assumed to be $\rG$-connections.

A connection $\nabla$ extends to an operator $d_\nabla\colon \Omega^p(E)\to\Omega^{p+1}(E),$ called the \emph{covariant exterior differentiation}  satisfying the Leibniz rule: 
$$ 
d_\nabla ( \alpha \otimes s )  =   d\alpha\otimes s +(-1)^p \alpha \wedge \nabla s,  
$$ 
where 
$\alpha \wedge \nabla s =\sum_i(\alpha\wedge e_i^\ast)\otimes \nabla_{e_i}s$, where 
$\{e_i^\ast\}$ is a local frame of $T^\ast M$. 

The $2$-form  $\rF_\nabla=(\rF_j^i )\in \Omega^2(\fg_E)$ defined locally by
$ 
\rF_j^i = d\rA ^i_j-\left( \sum_k \rA ^k_j\wedge \rA ^i_k \right)  
$ 
is called the \emph{curvature form} of $\nabla$. Since we identify $\fg_E$ with a subbundle of $\End(E)$, a connection $\nabla\in \cA(E)$ induces a connection on $\fg_E$ (still denoted by the same symbol):
\begin{equation}
\label{eq:End connexction}    
(\nabla\phi)(s)=\nabla(\phi(s))-\phi(\nabla s)\quad\text{for all}\quad\phi\in\Omega^0(\fg_E)\qandq s\in\Omega^0(E) .
\end{equation} 
If the section $\phi\in \Omega^0(\fg_E)$ has the local expression $\phi=\sum_{i,j}\phi^i_j s_i\otimes s_j^\ast$, where $\{s_i^\ast\}$ is the dual frame  of $\{s_i\}$, then $\nabla\phi\in \Omega^1(\End(E))$ is given by 
$$ 
(\nabla\phi)^i_j=d(\phi^i_ j)+[\phi,\rA]^i_j,
$$ 
where $\rA$ is the connection matrix \eqref{eq:connection matrix} and the bi-linear map 
$$ 
[\cdot\wedge\cdot]\colon \Omega^p(\fg_E)\times\Omega^q(\fg_E)\to \Omega^{p+q}(\fg_E) 
$$
is defined by 
\begin{equation}
    \label{eq:bracket} 
[\phi\wedge\psi]^i_j=\sum_h\left(\phi^h_j\wedge\psi^i_h-(-1)^{pq}\psi^h_j\wedge\phi^i_ h \right).
\end{equation} 
With this notation, the $\fg_E$-valued connection $\nabla$ in \eqref{eq:End connexction} can be naturally  extended to a covariant exterior differentiation $d_\nabla\colon\Omega^p(\fg_E)\to \Omega^{p+1}(\fg_E)$,  which is locally written as  
$$ 
(d\phi)^i_j=d(\phi^i_j)+(-1)^p[\phi\wedge \rA]^i_j. 
$$ 
From the structural equation $\rF_\nabla=d\rA-\frac{1}{2}[\rA\wedge\rA]$ and     properties of \eqref{eq:bracket}, the Ricci and Bianchi identities hold
\begin{equation}
    \label{eq:d2}
    d_\nabla\circ d_\nabla (\varphi )=[\varphi\wedge \rF_\nabla],\quad\text{and} \quad d_\nabla \rF_\nabla=0.
\end{equation}
We adopt the following convention for components of forms: Let $\varphi\in \Omega^p(\fg_E)$ be a $\fg_E$-valued $p$-form, then we write
\begin{align}\label{eq:phi expresion 1}
    \varphi &=\frac{1}{p!} \sum_{ij, k_1\dots k_p}{\varphi_{k_1\dots k_p}}^i_j dx^{k_1}\wedge dx^{k_2}\wedge\cdots\wedge dx^{k_p}\otimes s_i\otimes s_j^\ast,
 \end{align}
with respect to a frame $\{s_i\}$ and its co-frame $\{s_i^\ast\}$ and local functions ${\varphi_{k_1\dots k_p}}^i_j$, skew-symmetric with respect to $k_1,\dots,k_p$.

Finally, the connection $\nabla$ defines a linear differential operator on the tensor bundle $\bigotimes^p T^\ast M\otimes \fg_E$, denoted by the same symbol: Let $\varphi=\sum\varphi_{k_1\dots k_p}dx^{k_1}\otimes \cdots\otimes dx^{k_p}\in \Gamma\left(\bigotimes^pT^\ast M\otimes \fg_E\right)$, then $\nabla\varphi$ is a section of $\bigotimes^{p+1}T^\ast M\otimes \fg_E$ given by
$$ 
\nabla\varphi= \sum_{k, k_1\dots k_p}\nabla_k\varphi_{k_1\dots k_p}dx^{k}\otimes dx^{k_1}\otimes \cdots\otimes dx^{k_p} 
$$  
where
$  
\nabla_k\varphi_{k_1\dots k_p} =\frac{\partial}{\partial x^k}(\varphi_{k_1\dots k_p})+[\varphi_{k_1\dots k_p}, \rA_k].
$

Hence, for any  $\varphi\in \Omega^p(\fg_E),$ written locally as in \eqref{eq:phi expresion 1}, the covariant derivative $d_\nabla\varphi$ is
$$
d_\nabla\varphi=\frac{1}{p!}\sum_{k,k_1\dots k_p}\nabla_k\varphi_{k_1\dots k_p}dx^{k}\wedge dx^{k_1}\wedge \cdots\wedge dx^{k_p}.
$$  
More generally, for a section $\varphi $ of $\bigotimes^pT^\ast M\otimes\fg_E$, we set 
$$  
[\nabla_i,\nabla_j]\varphi_{k_1\dots k_p} =\nabla_i\nabla_j\varphi_{k_1\dots k_p}-\nabla_i\nabla_j\varphi_{k_1\dots k_p}.    
$$   
The curvature is locally represented by 
$ \rF_\nabla=\frac{1}{2}\sum_{ij}\rF_{ij}dx^i\wedge dx^j,$
where the coefficients $\rF_{ij}\in\fg$ of the curvature satisfy 
\begin{equation}\label{eq:coefficient curvature}
\rF_{ij}=-\rF_{ji}=\frac{\partial}{\partial x^i}\rA_j-\frac{\partial}{\partial x^j}\rA_i-[\rA_i,\rA_j].
\end{equation}
For a $p$-form $\varphi$ written as in \eqref{eq:phi expresion 1}, we have the formula
\begin{equation}\label{eq:RicciFormula}
[\nabla_i,\nabla_j]\varphi_{k_1\dots k_p}=[\varphi_{k_1\dots k_p},\rF_{ji}].
\end{equation} 
We denote by $\wt{\nabla}$ the sum of the Levi-Civita connection $D$ acting on $\Omega^{p}(M)$ and the connection $\nabla$ acting on $\Gamma(\fg_E)$, i.e. 
\begin{equation}\label{eq:nabla tilde}
\wt{\nabla}_i\varphi_j=D_i\varphi_j+[\varphi_j,\rA_i]=\frac{\del}{\del x^i}\varphi_j-\sum_k\Gamma^k_{ij}\varphi_k+[\varphi_j,\rA_i] .
\end{equation}

\subsection{The inner-product on \texorpdfstring{$\Omega^{p,0}(\fg_E)$}{lg}}
\label{sec:inner product}
Because of Lemma~\ref{lem:SDcharacterization}, we need to describe the global inner-product on $(p,0)$-forms for $p=0,1,2$ in details.  

First, we denote elements in $\Omega^{p,q}(\fg_E)$ by $\varphi^{p,q}$ to emphasize their bi-degree $(p,q)$ (cf. Appendix \ref{sec:tKahlerGeometry}). Second, for sets of indices   $I:=\{\mu_1<\cdots<\mu_p\}$ and $J:=\{\nu_1<\cdots<\nu_q\}$, we write 
$$ 
\varphi=\sum \varphi_{I,J}dz^I\wedge d\zbar^J, \qwhereq dz^I=dz^{\mu_1}\wedge\cdots\wedge dz^{\mu_p}\qandq  d\zbar^J=d\zbar^{\nu_1}\wedge\cdots\wedge d\zbar^{\nu_q}.
$$ 
The wedge product $\wedge\colon \Omega^k(TM^\ast)\times\Omega^r(TM^\ast)\to \Omega^{k+r}(TM^\ast)$ extends to $\fg_E$-valued forms   
as follows
$$ 
(\alpha,\beta)\in \Omega^k(\fg_E)\times\Omega^r(\fg_E)\xrightarrow[]{\wedge} (\alpha^i\wedge \beta^j)\otimes(e_i\otimes e_j)\xrightarrow[]{1\otimes[\cdot,\cdot]} (\alpha^i\wedge \beta^j)\otimes[ e_i,e_j] \in \Omega^{k+r}(\fg_E)
$$
We also extend complex conjugation to $\Omega^{p,q}(\fg_E)$ so that $\ol{\Omega^{p,q}(\fg_E)}=\Omega^{q,p}(\fg_E)$ by conjugating the form part of $\varphi \in \Omega^{p,q}(\fg_E)$. 

The Hodge star operator $\ast\colon \Lambda^k(T^\ast M)\to\Lambda^{n-k}(T^\ast M)$ extends to an operator on $E$-valued forms $\ast\colon \Omega^k(E)\to  \Omega^{k+1}(E)$ acting trivially on the endomorphism part, so $\ast(\alpha\otimes A)=(\ast \alpha)\otimes A.$ 

Analogously, we extend the transverse Hodge star operator $\ast_T\colon \Omega^k_H(M)\to \Omega^{n-1-k}_H(M)$, related to the usual Hodge star operator by 
\begin{equation}\label{eq:starT}
 \ast_T(\alpha )=(-1)^ki_\xi(\ast \alpha) 
\end{equation} 
where $i_\xi\alpha$ is the contraction of $\alpha $ by $\xi$.
We are supposing $\rG$ to be a compact Lie group, so that   $\fg$ admits some ad-invariant inner product $B(\cdot,\cdot)$. In fact, since  $G$ is supposed to be  compact semi-simple Lie group, there is a canonical choice of $B(\cdot,\cdot)$, namely the Cartan–Killing form of $\fg$, we assume this form now.  
Assume $M$ is compact, then the global $L^2$-inner product on $\Omega^p(\fg_E)$ is defined  by:
\begin{equation}\label{eq:inner product}
\left(\varphi, \psi \right)=-\int_M B(\varphi\wedge\ast \psi)=-\int_M \langle\varphi ,\psi \rangle\dvol ,  
\end{equation}
where $\langle\varphi ,\psi \rangle$ is obtained by tensoring the inner-product induced from $\Lambda^k \rT M_\C^\ast$ on the form part and $B(\cdot,\cdot)$ on the endomorphism part.

Moreover, using the Sasakian structure $\cS$ and \eqref{eq:starT}, the inner-product  \eqref{eq:inner product} can be re-written
$$
\left(\varphi, \psi \right)= -\int_M B(\varphi\wedge\ast_T\psi)\wedge \eta .
$$

Considering Lemma \ref{lem:SDcharacterization} and the splitting of $2$-forms in \eqref{eq:decompositionPQ}, we give an expression for the inner-product \eqref{eq:inner product}
 between $(p,0)$-forms  
$$ 
\varphi^{p,0}=\frac{1}{p!}\sum\varphi_{I}dz^I \quad \mbox{and} \quad
\psi^{p,0}=\frac{1}{p!}\sum\psi_{J}dz^J.
$$   
Ror $1$-form $\varphi =\varphi^{1,0}+\ol{\varphi^{1,0}}$ and $\psi=\psi^{1,0}+\ol{\psi^{1,0}}$, with $\varphi^{1,0},\psi^{1,0} \in\Omega^{1,0}_H(\fg_E)$, we have
\begin{align*}
    \langle \varphi,\psi \rangle 
      &= \langle \varphi^{1,0}+\ol{\varphi^{1,0}},\psi^{1,0}+\ol{\psi^{1,0}}\rangle 
       = \langle \varphi^{1,0},\psi^{1,0}\rangle +\ol{\langle \varphi^{1,0 }, \psi^{1,0} \rangle}.
\end{align*}
Hence the global inner-product is computed to be
\begin{equation}\label{eq:inner 1,0}
\left(\varphi  ,\psi \right) =-2 \int_M \Re\langle \varphi^{1,0},\psi^{1,0}\rangle\dvol.
\end{equation}
While for $2$-forms 
$$\varphi =\varphi^{2,0}+\ol{\varphi^{2,0}}+\omega\otimes\varphi^0
\quad \mbox{and}\quad  
\psi =\psi^{2,0}+\ol{\psi^{2,0}}+\omega\otimes\psi^0
$$ 
with $\varphi^{2,0},\psi^{2,0}\in\Omega^{2,0}_H(\fg_E)$ and $ \psi^{ 0}\in\Omega^{ 0}_H(\fg_E)$,  the inner product \eqref{eq:inner product} becomes
\begin{equation}\label{eq:inner 2,0}
\left(\varphi  ,\psi \right)=-2\int\left(\Re\langle \varphi^{2,0},\psi^{2,0}\rangle+\langle\varphi^0,\psi^0\rangle\right)\dvol.
\end{equation}

\section{A Weitzenböck formula for \texorpdfstring{$(2,0)$}{Lg}-forms on a Sasakian
\texorpdfstring{$7$}{Lg}-manifold}\label{sec:weitzenbock} 

This section applies Bochner-type arguments to prove the vanishing Theorem \ref{thm:main theorem}. 
The general principle is: Denote by $\bar{\partial}^\ast$ and $\nabla^\ast$ the adjoint operators of $\bar{\partial}$ and $\nabla$, respectively. The difference $R=\nabla^\ast \nabla-2(\bar{\partial}^\ast\bar{\partial}+\bar{\partial} \bar{\partial}^\ast)$ is a zero-order operator depending on the curvature of the connection 
\cite[Chapter~2]{lawson2016spin}:
\begin{equation}
    \label{eq:Weitzenbock geral}
    \nabla^\ast \nabla= 2(\bar{\partial}^\ast\bar{\partial}+\bar{\partial} \bar{\partial}^\ast)+\rR\in \Gamma(\End(\Lambda^{p,q}\otimes E))
\end{equation}
If $\rR$ is positive definite on $\Lambda^{p,0}\otimes E$, then one can show that there is no holomorphic section in $\Lambda^{p,0}\otimes E$, by applying  \eqref{eq:Weitzenbock geral} to $\sigma\in \Gamma(\Lambda^{p,0}\otimes E)$, taking the inner-product with $\sigma$, integrating by parts, and using that $\bar{\partial}$ vanishing on $\Lambda^{p,0}\otimes E$. We obtain a version of \eqref{eq:Weitzenbock geral} on a Sasakian $7$-manifold $ (M,\cS)$ and transverse  Kähler form $\omega:=d\eta$, and use  it to show the vanishing of the second cohomology group of the complex \eqref{eq:Complexd7}.\\

Let      
$ 
\Delta_{\partial_\nabla} =\partial_\nabla\partial^{\ast}_\nabla+\partial^{\ast}_\nabla\partial_\nabla    
$ be the generalized Dolbeault Laplacian, for $\varphi \in\Omega^{2,0}(\fg_E)$
$$ 
\Delta_{\partial_\nabla}\varphi=\frac12\sum_{\mu\nu}\left( \Delta_{\partial_\nabla}\varphi \right)_{\mu\nu}dz^\mu\wedge dz^\nu\in \Omega^{2,0}(\fg_E),
$$ 
where   
$  \left( \Delta_{\partial_\nabla}\varphi \right)_{\mu\nu}:=
\Big(\left(\partial_{\nabla}\partial^{\ast}_\nabla+\partial^{\ast}_\nabla\partial_{\nabla}\right)\varphi \Big)_{\mu\nu} $ is a $\fg$-valued function, skew-symmetric in $\mu,\nu$.

The operator $\partial_\nabla$ acts on $\Omega^{1,0}(\fg_E)$ as:
\begin{equation}\label{eq:partial on 1,0}
\begin{array}{l}
\partial_\nabla\varphi 
=\frac{1}{2}\sum_{\mu\nu}(\partial_\nabla\varphi )_{\mu\nu}dz^\mu\wedge dz^\nu,  \qwhereq   (\partial_\nabla\varphi )_{\mu\nu} =\wt{\nabla}_\mu\varphi_\nu-\wt{\nabla}_\nu\varphi_\mu,
\end{array}
\end{equation}
where the covariant derivative $\wt{\nabla}$ was defined in \eqref{eq:nabla tilde}, and on $\Omega^{2,0}(\fg_E)$ by 
\begin{equation}\label{eq:partial on 2,0}
\begin{array}{l}
\partial_\nabla\varphi=\frac{1}{3!}\sum_{\mu\nu\sigma}(\partial_\nabla\varphi)_{\mu\nu\sigma}dz^\mu\wedge dz^\nu\wedge dz^\sigma, \qwhereq  (\partial_\nabla\varphi)_{\mu\nu\sigma} =
\wt{\nabla}_{\mu}\varphi_{\nu\sigma}+
\wt{\nabla}_{\nu}\varphi_{\sigma\mu}+
\wt{\nabla}_{\sigma}\varphi_{\mu\nu}.  
\end{array}
\end{equation}
The formal adjoint of $\partial_\nabla$ operates on $\Omega^{2,0}(\fg_E)$ by 
\begin{equation}\label{eq:partial ad on 2,0}
\begin{array}{l}
\partial^{\ast}_\nabla\varphi
=\sum_{\mu} (\partial^{\ast}_\nabla\varphi )_{\mu}dz^\mu  
\qwhereq  (\partial^{\ast}_\nabla\varphi )_{\mu}
=-\sum_{\nu\tau}g^{\nu\bar{\tau}}\wt{\nabla}_{\bar{\tau}}\varphi_{\nu\mu}  
\end{array}
\end{equation}
and on $\Omega^{3,0}(\fg_E)$ as 
\begin{equation}\label{eq:partial ad on 3,0} 
\partial^{\ast}_\nabla\varphi  
=\frac{1}{2}\sum_{\mu\nu}(\partial_\nabla^\ast\varphi)_{\mu\nu}dz^\mu\wedge dz^\nu
\qwhereq  (\partial_\nabla^\ast\varphi)_{\mu\nu}
=-\sum_{\sigma\tau}g^{\sigma\bar{\tau}}\wt{\nabla}_{\bar{\tau}}\varphi_{\sigma\mu\nu} 
\end{equation}
\begin{proposition}[Weitzenböck   formula]
\label{prop:bochner}
Let $(M,\cS)$ be  a compact, connected, Sasakian  $7$-manifold, $E\to M$ a Sasakian vector bundle, $\nabla\in \cA(E) $ and  $\rF_\nabla$ its curvature. In transverse coordinates (cf. Appendix \ref{sec:tKahlerGeometry}), denote by  
$\varphi = \sum_{\gamma,\tau}\varphi_{\gamma\tau}dz^\gamma\wedge dz^\tau \in \Omega^{2,0}(\fg_E )$ (where $\varphi_{\gamma\tau}$ are $\fg$-valued functions),    
then we have the following  formula
\begin{equation} 
    \label{eq:weitzenbock}
\left( \Delta_{\partial_\nabla}\varphi \right)_{\mu\nu}
=-\sum_{\alpha\beta} g^{\alpha\bar{\beta}}\widetilde{\nabla}_{\bar{\beta}}\widetilde{\nabla}_{\alpha}\varphi_{\mu\nu} 
  - {(\cF\varphi)}_{\mu\nu} - {(\cR\varphi)}_{\mu\nu},  
\end{equation}
here    
$   
{(\cF\varphi)}_{\mu\nu} =\sum_{\alpha\beta} g^{\alpha\bar{\beta}} \left([\varphi_{\alpha\nu},\rF_{\mu\bar{\beta}}]
-[\varphi_{\alpha\mu},\rF_{\nu\bar{\beta}}]\right)   
$ 
and  
$ 
{(\cR\varphi)}_{\mu\nu} =\sum_{\alpha\beta} g^{\alpha\bar{\beta}}\left(\rR_{\bar{\beta}\mu } \varphi_{\alpha\nu}  - \rR_{\bar{\beta} \nu} \varphi_{\alpha\mu}\right)    
$  
define endomorphisms  $\cF$  and $\cR$ on $\Omega^{2,0}(\fg_E)$. They depend on the curvature $F_\nabla$ and the transverse Ricci curvature $\rR_{\bar{\alpha}\beta}=\sum g^{\bar{\mu}\nu}\rR_{\bar{\nu}\mu\bar{\alpha}\beta}$ (cf. Definition \ref{def:RicciT}).
\end{proposition}
\begin{proof}(Proposition ~\ref{prop:bochner})\\
Let $\varphi \in \Omega^{2,0}(\fg_E)$ written in transverse coordinates as $\varphi =\frac12\sum \varphi_{\gamma\tau}dz^\gamma\wedge dz^\tau $. From \eqref{eq:partial on 1,0} and \eqref{eq:partial ad on 2,0} we have
\begin{align*}
(\partial_\nabla\partial^{\ast}_\nabla\varphi )_{\mu\nu}
&=\wt{\nabla}_\mu(\partial^{\ast}_\nabla\varphi)_\nu-\wt{\nabla}_\nu (\partial^{\ast}_\nabla\varphi)_\mu  
 =\wt{\nabla}_\mu \Big(-\sum_{\alpha\beta}g^{\alpha\bar{\beta}}\wt{\nabla}_{\bar{\beta}}\varphi_{\alpha\nu}\Big)+\wt{\nabla}_\nu \Big( \sum_{\alpha\beta}g^{\alpha\bar{\beta}}\wt{\nabla}_{\bar{\beta}}\varphi_{\alpha\mu}\Big) \\
&= -\sum_{\alpha\beta}g^{\alpha\bar{\beta}}\left( \wt{\nabla}_\mu\wt{\nabla}_{\bar{\beta}}\varphi_{\alpha\nu}
   -\wt{\nabla}_\nu\wt{\nabla}_{\bar{\beta}}\varphi_{\alpha\mu}\right).
\end{align*}
Symmetrically, with \eqref{eq:partial ad on 3,0} and \eqref{eq:partial on 2,0}, we obtain  
\begin{align*}
(\partial^{\ast}_\nabla\partial_\nabla\varphi)_{\mu\nu}
&=-\sum_{\alpha\beta}g^{\alpha\bar{\beta}}\wt{\nabla}_{\bar{\beta}}(\partial_\nabla\varphi)_{\alpha\mu\nu}   
 = -\sum_{\alpha\beta}g^{\alpha\bar{\beta}}\wt{\nabla}_{\bar{\beta}}
\Big(\wt{\nabla}_{\alpha}\varphi_{\mu\nu}+ 
  \wt{\nabla}_{\mu}\varphi_{\nu\alpha}+
  \wt{\nabla}_{\nu}\varphi_{\alpha\mu}\Big) \\
&=-\sum_{\alpha\beta}g^{\alpha\bar{\beta}}\left(\wt{\nabla}_{\bar{\beta}} 
   \wt{\nabla}_{\alpha}\varphi_{\mu\nu}+\wt{\nabla}_{\bar{\beta}}\wt{\nabla}_{\mu}\varphi_{\nu\alpha}
  +\wt{\nabla}_{\bar{\beta}}\wt{\nabla}_{\nu}\varphi_{\alpha\mu}\right)  .
\end{align*}
Now, summing up we obtain
$
(\Delta_{\partial_\nabla}\varphi)_{\mu\nu} 
$ 
   
\begin{align*}
(\Delta_{\partial_\nabla}\varphi)_{\mu\nu}
&=-\sum_{\alpha\beta}g^{\alpha\bar{\beta}}\left\{\left( \wt{\nabla}_\mu\wt{\nabla}_{\bar{\beta}}\varphi_{\alpha\nu}
      - \wt{\nabla}_\nu\wt{\nabla}_{\bar{\beta}}\varphi_{\alpha\mu}\right)
      +  \left(\wt{\nabla}_{\bar{\beta}}\wt{\nabla}_{\alpha}\varphi_{\mu\nu}
      +\wt{\nabla}_{\bar{\beta}}\wt{\nabla}_{\mu}\varphi_{\nu\alpha}
      +\wt{\nabla}_{\bar{\beta}}\wt{\nabla}_{\nu}\varphi_{\alpha\mu}\right) \right\}\\
&=-\sum_{\alpha\beta}g^{\alpha\bar{\beta}}\left\{\wt{\nabla}_{\bar{\beta}} \wt{\nabla}_{\alpha}\varphi_{\mu\nu}+ \wt{\nabla}_\mu\wt{\nabla}_{\bar{\beta}}\varphi_{\alpha\nu}
     - \wt{\nabla}_{\bar{\beta}}\wt{\nabla}_{\mu}\varphi_{\alpha\nu}
     +\wt{\nabla}_{\bar{\beta}}\wt{\nabla}_{\nu}\varphi_{\alpha\mu}
     - \wt{\nabla}_\nu\wt{\nabla}_{\bar{\beta}}\varphi_{\alpha\mu}  \right\}\\ 
&=-\sum_{\alpha\beta}g^{\alpha\bar{\beta}}\wt{\nabla}_{\bar{\beta}}\wt{\nabla}_{\alpha}\varphi_{\mu\nu}
    -\underbrace{\sum_{\alpha\beta}g^{\alpha\bar{\beta}}\left( [\wt{\nabla}_{\mu},\wt{\nabla}_{\bar{\beta}}]\varphi_{\alpha\nu}
    -[\wt{\nabla}_{\nu},\wt{\nabla}_{\bar{\beta}}]\varphi_{\alpha\mu}\right)}_{(a)}
\end{align*}
The terms in $(a)$ describe the curvature contributions from both the bundle and the base manifold. We use the formulas \eqref{eq:RicciFormula} and \eqref{eq:nabla tilde} to compute a Ricci formula for the connection $\wt{\nabla}$
\begin{align*}
[\wt{\nabla}_{\mu},\wt{\nabla}_{\bar{\beta}}]\varphi_{\alpha\nu}  &= [D_\mu,D_{\bar{\beta}}] \varphi_{\alpha\nu} +[\varphi_{\alpha\nu},\rF_{\mu\bar{\beta}}] \\
&= R^\epsilon_{\mu\bar{\beta}\alpha}\varphi_{\epsilon\nu}+R^{\epsilon}_{\mu\bar{\beta}\nu}\varphi_{\alpha\epsilon}+[\varphi_{\alpha\nu},\rF_{\mu\beta}] 
\end{align*} 
(alternatively, this can be seen as the curvature of $\wt{\nabla}=D\otimes 1+1\otimes \nabla$ on $\End(\fg_E)$).
The expression $(a)$ can then be expressed in terms of the curvatures of the Levi-Civita connection and $\nabla$:
\begin{align*}
(a)&=
  \sum_{\alpha\beta}g^{\alpha\bar{\beta}}\left( 
  [D_\mu,D_{\bar{\beta}}] \varphi_{\alpha\nu} +[\varphi_{\alpha\nu},\rF_{\mu\bar{\beta}}]
-\left([D_\nu,D_{\bar{\beta}}] \varphi_{\alpha\mu} +[\varphi_{\alpha\mu},\rF_{\nu\bar{\beta}}] \right) \right) \\ 
&=\sum_{\alpha\beta} g^{\alpha\bar{\beta}} \left(\rR_{\mu\bar{\beta}\alpha}^{\epsilon}\varphi_{\epsilon\nu}
+ \rR_{\mu\bar{\beta}\nu}^{\epsilon} \varphi_{\alpha\epsilon} 
+  [\varphi_{\alpha\nu},\rF_{\mu\bar{\beta}}] - \rR_{\nu\bar{\beta}\alpha}^{\epsilon}\varphi_{\epsilon\mu}
     -\rR_{\nu\bar{\beta}\mu}^{\epsilon} \varphi_{\alpha\epsilon}-[\varphi_{\alpha\mu},\rF_{\nu\bar{\beta}}] \right)    \\
 &= \sum_{\alpha\beta} g^{\alpha\bar{\beta}} \left(             
       [\varphi_{\alpha\nu},\rF_{\mu\bar{\beta}}]
      -[\varphi_{\alpha\mu},\rF_{\nu\bar{\beta}}]\right) 
   + \underbrace{ 
  \sum_{\alpha\beta}g^{\alpha\bar{\beta}}\left(\rR_{\mu\bar{\beta}\alpha}^{\epsilon}\varphi_{\epsilon\nu}
        +\rR_{\mu\bar{\beta}\nu}^{\epsilon} \varphi_{\alpha\epsilon}
        -\rR_{\nu\bar{\beta}\alpha}^{\epsilon}\varphi_{\epsilon\mu}
        -\rR_{\nu\bar{\beta}\mu}^{\epsilon}
                      \varphi_{\alpha\epsilon} \right)  
   }_{(b)}.
\end{align*}
Finally, we write the term $(b)$ in terms of the transverse Ricci curvature  
$
\rR_{\bar{\alpha}\beta}=\sum g^{\mu\bar{\nu}}\rR_{\bar{\nu}\mu\bar{\alpha}\beta}
$,
using the symmetries of the curvature tensor:
\begin{align*}
(b) &=   \sum_{\alpha\beta\epsilon}\left(
            g^{\alpha\bar{\beta}}  \rR_{\mu\bar{\beta}\alpha}^\epsilon\varphi_{ \epsilon\nu}
           -g^{\alpha\bar{\beta}}  \rR_{\nu\bar{\beta}\alpha}^{\epsilon}\varphi_{\epsilon\mu}
           +g^{\alpha\bar{\beta}}( \rR_{\mu\bar{\beta}\nu}^{\epsilon} 
           -\rR_{\nu\bar{\beta}\mu}^{\epsilon})\varphi_{\alpha\epsilon}\right)\\
    &= \sum_{\alpha\beta\epsilon\tau}\left(
                g^{\epsilon\bar{\tau}}g^{\alpha\bar{\beta}} \rR_{\bar{\tau}\mu\bar{\beta}\alpha} \varphi_{\epsilon\nu}
               -g^{\epsilon\bar{\tau}}g^{\alpha\bar{\beta}}\rR_{\bar{\tau}\nu\bar{\beta}\alpha} \varphi_{\epsilon\mu}
               +g^{\alpha\bar{\beta}}g^{\epsilon\bar{\tau}}( \rR_{\bar{\tau}\mu\bar{\beta}\nu}  
               -\rR_{\bar{\tau}\nu\bar{\beta}\mu} )\varphi_{\alpha\epsilon}\right)\\
    &= \sum_{\tau\epsilon}\left(
                 g^{\epsilon\bar{\tau}} \rR_{\bar{\tau}\mu } \varphi_{\epsilon\nu}
                -g^{\epsilon\bar{\tau}} \rR_{\bar{\tau}\nu } \varphi_{\epsilon\mu} \right) .
\end{align*}
We used that  $\rR_{\bar{\tau}\mu\bar{\beta}\nu}
-\rR_{\bar{\tau}\nu\bar{\beta}\mu}=0$ since the metric is transversely Kähler \cite[Proposition~6.3]{morrow2006complex}. The Weitzenböck formula \eqref{eq:weitzenbock} then follows by defining operators $\cF$ and $\cR$ on $\Omega^{2,0}(\fg_E) $ as
\begin{align*}
\cF  :\Omega^{2,0}(\fg_E) &\to \Omega^{2,0}(\fg_E) &\\
          \varphi         &\mapsto\cF\varphi:= \sum_{\mu\nu} {(\cF\varphi)}_{\mu\nu}dz^\mu\wedge dz^\nu \\
    \text{with}  \quad    & {(\cF\varphi)}_{\mu\nu} =\sum_{\alpha\beta} g^{\alpha\bar{\beta}} \left([\varphi_{\alpha\nu},\rF_{\mu\bar{\beta}}]  -[\varphi_{\alpha\mu},\rF_{\nu\bar{\beta}}]\right)
\end{align*}
and
\begin{align*}
\cR : \Omega^{2,0}(\fg_E) &\to                 \Omega^{2,0}(\fg_E) &\\
          \varphi             &\mapsto   \cR\varphi:= \sum_{\mu\nu} {(\cR\varphi)}_{\mu\nu}dz^\mu\wedge dz^\nu \\
 \text{with}  \quad   &{(\cR\varphi)}_{\mu\nu}= \sum_{\alpha\beta} g^{\alpha\bar{\beta}}\left(\rR_{\bar{\beta}\mu } \varphi_{\alpha\nu}  
                        -\rR_{\bar{\beta} \nu} \varphi_{\alpha\mu}\right). 
\end{align*}
\end{proof}
Despite the proof of the following Lemma is is  straightforward from the definition of $\cF$, we declare it with out proof because we do not used it in this paper
\begin{lemma}
\label{lem:fR selfadjoint}
The operator $\cF$ is self-adjoint on $\Omega^{2,0}(\fg_E)$.
\end{lemma}  
\section{The vanishing theorem}
\label{sec:vanishing} 

This section proves the \textit{Vanishing Theorem} \ref{thm:main theorem}. \\
The Bochner method consists in applying Formula \eqref{eq:weitzenbock} and then integrate by parts. For the sake of ease, we divide the proof in steps and start with preliminary results on the operator $d_7^\ast$ acting on $\Omega_1(\fg_E)$ and its formal adjoint. To distinguish from summative indices, we denote the imaginary unit  $\ii\in \C$  in bold font.

\begin{lemma}
\label{lem:previo 1}
Let $\varphi^0\otimes \omega $ be an element in $\Omega_1^2(\fg_E)$, where $\omega$ is  the transverse  Kähler form. Denote by $d_7^\ast$ the formal adjoint of the operator $d_7$, with respect to the inner product \eqref{eq:inner product}. Then
\begin{equation}\label{eq:previo 1}
d_7^\ast(\omega\otimes\varphi^0)=-2\ii(\partial_\nabla\varphi^0-\ol{\partial}_\nabla\varphi^0).
\end{equation} 
\end{lemma}
\begin{proof}
Let $\psi=\psi^{1,0}+\ol{\psi^{1,0}}\in\Omega^1_H(\fg_E)$ be a $1$-form, then, by definition,
$$ 
\Big( d_7^\ast(\omega\otimes\varphi^0),\psi\Big)=\Big( \omega\otimes\varphi^0,d_7\psi\Big).
$$ 
To compute the right-hand side, note that $\omega\otimes\varphi^0$ is of type $(1,1)$ (Proposition~\ref{prop: omega (1,1)}), so we only have to consider the $(1,1)$-part of 
$d_7\psi$ which is from \eqref{eq:def d7} given by 
$$
\omega\otimes\Big\langle\ol{\partial}_\nabla\psi^{1,0}+\partial_\nabla\ol{\psi^{1,0}},\omega \Big\rangle
$$
In transverse coordinates,  
we express   
$ 
\psi^{1,0}=\sum_{\gamma}\psi_\gamma dz^\gamma ,
$  therefore
\begin{align*}
\left < \partial_\nabla\ol{\psi^{1,0}} ,\omega\right > 
&=\left < \partial_\nabla \left(\sum_{\gamma}\ol{\psi}_\gamma d\zbar^\gamma\right)  ,\ii\sum_{\mu\nu}g_{\mu\bar{\nu}}dz^\mu\wedge d\zbar^\nu \right > 
= -\ii\left <\sum_{\gamma\kappa}\nabla_\kappa\ol{\psi}_\gamma dz^\kappa\wedge d\zbar^\gamma  , \sum_{\mu\nu}g_{\mu\bar{\nu}}dz^\mu\wedge d\zbar^\nu \right >\\
&=-\ii\sum_{\gamma\kappa\mu\nu}g_{\nu\bar{\mu}}g^{\kappa\bar{\mu}} g^{\nu\bar{\gamma}} \nabla_\kappa\ol{\psi}_\gamma  
=-\ii\sum_{\gamma\kappa\nu}\delta_{\nu}^{\kappa} g^{\nu\bar{\gamma}}  \nabla_\kappa\ol{\psi}_\gamma
= -\ii\sum_{\gamma \nu}  g^{\nu\bar{\gamma}}\nabla_\nu\ol{\psi}_\gamma   \\ 
&=-\ii\sum_{\mu\nu }  g^{\mu\bar{\nu}}\nabla_\mu\ol{\psi}_\nu    .      
\end{align*}
Analogously, we compute 
\begin{align*}
\left <\bar{\partial}_\nabla\psi^{1,0},\omega\right > 
      &=\left <\bar{\partial}_\nabla \left(\sum_{\gamma}\psi_\gamma dz^\gamma \right),\ii\sum_{\mu\nu}g_{\mu\bar{\nu}}dz^\mu\wedge d\zbar^\nu  \right > 
       =-\ii  \left < \sum_{\kappa\gamma}\nabla_ {\bar{\kappa}}\psi_\gamma d\zbar^\kappa\wedge   dz^\gamma, \sum_{\mu\nu}g_{\mu\bar{\nu}}dz^\mu\wedge d\zbar^\nu  \right > \\
      &= \ii\sum_{\kappa\gamma\mu\nu}g_{\nu\bar{\mu}}g^{\gamma\bar{\mu}}g^{\nu\bar{\kappa}} \nabla_{\bar{\kappa}}\psi_\gamma  
       =\ii\sum_{\kappa\gamma \nu}\delta_{\nu}^{\gamma}g^{\nu\bar{\kappa}} \nabla_{\bar{\kappa}}\psi_\gamma 
       =\ii\sum_{\kappa \nu}g^{\nu\bar{\kappa}} \nabla_{\bar{\kappa}}\psi_\nu\\
      &=\ii\sum_{\nu \mu}g^{\mu\bar{\nu}} \nabla_{\bar{\nu}}\psi_\mu .
\end{align*} 
Hence  
\begin{align*}
\Big(\omega\otimes\varphi^0,d_7\psi\Big) &= \Big(\omega\otimes\varphi^0,\omega\otimes\langle\ol{\partial}_\nabla\psi^{1,0}+\partial_\nabla\ol{\psi^{1,0}},\omega \rangle\Big)\\
&= -2\ii\displaystyle\int_M B\left(\varphi^0,\sum_{\mu\nu}g^{\mu\bar{\nu}}\Big(\nabla_{\mu}\ol{\psi}_\nu-\nabla_{\bar{\nu}}\psi_\mu\Big) \right)\dvol .
\end{align*} 
By Stokes Theorem we have, 
$$
\displaystyle\int_M\sum_{\mu\nu} D_{\bar{\mu}}\left(g^{\mu\bar{\nu}}B(\varphi^0,\psi_\nu \right))\dvol =0,
$$
so
$$
\int_M B\left(\varphi^0, \sum_{\mu\nu}g^{\nu\bar{\mu}}D_{\bar{\mu}}\psi_\nu\right)\dvol =-\int_M B\left(\sum_{\mu\nu}g^{\nu\bar{\mu}}D_{\bar{\mu}}\varphi^0, \psi_\nu\right)\dvol .
$$
If $\rA^{1,0}$ and $\ol{\rA^{1,0}}$ denote the $(1,0)$ and $(0,1)$ parts of the connection form $\rA$, then
\begin{align*}
&\int_M\sum_{\mu\nu}g^{\mu\bar{\nu}}B\left(\varphi^0,\nabla_{\mu}\ol{\psi}_\nu\right) \dvol
= \int_M\sum_{\mu\nu}g^{\mu\bar{\nu}}B(\varphi^0, D_{\mu}\ol{\psi}_\nu+[\ol{\psi}_\nu, \rA_\mu^{1,0}] )\dvol\\
&=\int_M\sum_{\mu\nu}g^{\mu\bar{\nu}}B(\varphi^0, D_{\mu}\ol{\psi}_\nu)+\sum_{\mu\nu}g^{\mu\bar{\nu}}B(\varphi^0,[\ol{\psi}_\nu, \rA_\mu^{1,0}])\dvol\\
&=\int_M-B\left(\sum_{\mu\nu}g^{\mu\bar{\nu}} D_\mu\varphi^0,\ol{\psi}_\nu \right)+\sum_{\mu\nu}g^{\mu\bar{\nu}}B\left(\varphi^0,[\ol{\psi}_\nu, \rA_\mu^{1,0}]\right)\dvol\\
&=\int_M-B\left( \sum_{\mu\nu}g^{\mu\bar{\nu}} \left(\nabla_\mu\varphi^0-   [\varphi^0,\rA^{1,0}_\mu]\right),\ol{\psi}_\nu\right)+\sum_{\mu\nu}g^{\mu\bar{\nu}}B\left(\varphi^0 , [\ol{\psi}_\nu, \rA_\mu^{1,0}]\right)\dvol\\
&=-\int_M B\left(\sum_{\mu\nu}g^{\mu\bar{\nu}}\nabla_\mu\varphi^0,\ol{\psi}_\nu\right)+\sum_{\mu\nu} g^{\mu\bar{\nu}}\left( B\left([\varphi^0,\rA^{1,0}_\mu],\ol{\psi}_\nu\right)+B\left(\varphi^0,[\ol{\psi}_\nu,\rA_\mu^{1,0}]\right)\right) \dvol\\
&=-\int_M B\left(\sum_{\mu\nu} g^{\mu\bar{\nu}}  \nabla_\mu\varphi^0,\ol{\psi}_\nu \right)\dvol,
\end{align*}
because $B$ is $\ad$-invariant so $B([\varphi^0,\rA^{1,0}_\mu],\ol{\psi}_\nu)+B(\varphi^0,[\ol{\psi}_\nu,\rA_\mu^{1,0}])=0$.

A similar computation shows that  
$$
\displaystyle\int_M\sum_{\mu\nu}g^{\mu\bar{\nu}}B(\varphi^0,\nabla_{\bar{\nu}}\psi_\mu)\dvol=-\int_MB \left(\sum_{\mu\nu}g^{\mu\bar{\nu}}\nabla_{\bar{\nu}}\varphi^0,\psi_\mu\right) \dvol .
$$
Observing that
\begin{align*}
\left(\omega\otimes\varphi^0,d_7\psi \right)
&=-2\ii\displaystyle\int_M\sum_{\mu\nu}g^{\mu\bar{\nu}}B(\varphi^0, \nabla_{\mu}\ol{\psi}_\nu-\nabla_{\bar{\nu}}\psi_\mu )\dvol\\
&=-2\ii\left(\displaystyle\int_M\sum_{\mu\nu}g^{\mu\bar{\nu}}B(\varphi^0, \nabla_{\mu}\ol{\psi}_\nu )\dvol- \displaystyle\int_M\sum_{\mu\nu}g^{\mu\bar{\nu}}B(\varphi^0,  \nabla_{\bar{\nu}}\psi_\mu )\dvol\right)\\
&=2\ii  \displaystyle\int_M B\left(\sum_{\mu\nu} g^{\mu\bar{\nu}}  \nabla_\mu\varphi^0,\ol{\psi}_\nu \right)\dvol-  2\ii \int_MB \left(\sum_{\mu\nu}g^{\mu\bar{\nu}}\nabla_{\bar{\nu}}\varphi^0,\psi_\mu\right) \dvol  \\
&=2\ii  \displaystyle\int_M \sum_{\mu\nu} g^{\mu\bar{\nu}}B\left(  \nabla_\mu\varphi^0,\ol{\psi}_\nu \right)\dvol-  2\ii \int_M \sum_{\mu\nu}g^{\mu\bar{\nu}}B \left(\nabla_{\bar{\nu}}\varphi^0,\psi_\mu\right) \dvol  \\ 
&=\left(2\ii(\partial_\nabla\varphi^0-\bar{\partial}_\nabla\varphi^0),\psi \right) ,
\end{align*} 
and Lemma~\ref{lem:previo 1} follows.
\end{proof}
Now, with Lemma \ref{lem:previo 1}, we can prove the 
following
\begin{lemma}\label{lem:previo 2}
Let $\omega\otimes\varphi^0\in\Omega^2_1(\fg_E)$, $\varphi^0\in \Gamma(\fg_E)$ and $\omega$ the transverse fundamental form, then  
$$
d_7\circ d_7^\ast(\omega\otimes\varphi^0)=-4\ii\omega\otimes\langle \partial_\nabla\circ\bar{\partial}_\nabla\varphi^0,\omega\rangle\in\Omega^2_1(\fg_E)
$$ 
\end{lemma}
\begin{proof}
From Formula \eqref{eq:d2} we know that 
$ 
(d_\nabla\circ d_\nabla)\varphi^0 =[\varphi^0\wedge \rF_\nabla],
$
is of type $(1,1)$. On the other hand 
\begin{equation}
    \label{eq:da2}
    (d_\nabla\circ d_\nabla)\varphi^0= \partial_\nabla\partial_\nabla\varphi^0  +\bar{\partial}_\nabla\bar{\partial}_\nabla\varphi^0
    +(\partial_\nabla\bar{\partial}_\nabla +\bar{\partial}_\nabla\partial_\nabla)\varphi^0
\end{equation}
so, comparing bidegrees, we infer that 
$$\partial_\nabla\partial_\nabla\varphi^0=
\bar{\partial}_\nabla\bar{\partial}_\nabla\varphi^0=0 \quad \mbox{and}\quad   
[\varphi^0\wedge \rF_\nabla]=(\partial_\nabla\bar{\partial}_\nabla +\bar{\partial}_\nabla\partial_\nabla)\varphi^0. 
$$
Applying the operator $d_7$ (cf. \eqref{eq:def d7}) and recalling that  $\rF_\nabla$ is of type $(1,1)$ and orthogonal to $\omega$  ([cf. )Proposition \ref{prop: selfdual = HYM}), yields
\begin{align*}
d_7(d_7^\ast(\omega\otimes \varphi^0))
  &=  2\ii d_7\left( \partial_\nabla\varphi^0-\ol{\partial}_\nabla\varphi^0\right)  \\
  &=  2\ii\left(\partial_\nabla\partial_\nabla\varphi^0
     +\omega\otimes\left<   \bar{\partial}_\nabla\partial_\nabla\varphi^0
     -\partial_\nabla\bar{\partial}_\nabla\varphi^0 ,\omega \right>
     -\bar{\partial}_\nabla\bar{\partial}_\nabla \varphi^0\right)\\
  &= 2\ii\omega\otimes\left< \bar{\partial}_\nabla\partial_\nabla\varphi^0
    -\partial_\nabla\bar{\partial}_\nabla\varphi^0 ,\omega \right>\\
  &= 2\ii\left(\omega\otimes\left<[\varphi^0\wedge \rF_\nabla] ,\omega \right>
    -2\omega\otimes\left< \partial_\nabla\bar{\partial}_\nabla\varphi^0,\omega \right>  \right)\\
  &= -4\ii\;\omega\otimes\left< \partial_\nabla\bar{\partial}_\nabla\varphi^0,\omega\right> ,
\end{align*} 
since $\left<[\varphi^0\wedge \rF_\nabla] ,\omega \right>=\left[\varphi^0,\left<\rF_\nabla,\omega\right> \right]=0$, 
as can be directly verified
\begin{align*}
\Tr([\varphi^0\wedge\rF_\nabla] \omega ) &= \sum_{j i}([\varphi^0\wedge\rF_\nabla]^j_i\wedge\omega^i_j) 
= \sum_{ji}\left(\sum_l\left(\varphi^{l}_{i} \rF^{j}_{l}
 -\varphi^{j}_{l}\rF^{\nabla l}_{i}\right)\wedge\omega^i_j\right)\\
&= \sum_{j i l}\left(\varphi^l_i\rF^{j}_{l}\wedge\omega^i_j
     -\varphi^{j}_{l}\omega^i_j\wedge\rF^{l}_{i}\right) 
    =\sum_{i l} \varphi^l_i(\rF_\nabla \omega)^i_l 
     -\sum_{j l} \varphi^j_l(\omega  \rF_\nabla)^l_j \\
   &=\sum_{i l}\left( \varphi^l_i(\rF_\nabla \omega)^i_l 
     - \varphi^i_l(\omega \rF_\nabla)^l_i \right) 
    = \Tr(\varphi^0\rF_\nabla \omega )-\Tr(\varphi^0  \omega \rF_\nabla)\\
    &=\left<\varphi^0,[\rF_\nabla,\omega] \right> =0.
\end{align*}
\end{proof}

Finally, we use  Lemma \ref{lem:previo 2}   and the Weitzenböck formula \eqref{eq:weitzenbock} to prove the 
\begin{theorem}[Vanishing Theorem]\label{thm:main theorem} 
Let $E\to M$ be a complex Sasakian vector bundle on a compact, connected Sasakian manifold  $(M,\cS)$ (cf. Definition \ref{thm:app sasakian manifold}) and $\nabla\in\cA(E)$ an irreducible SDCI \eqref{eq: instanton} 
with curvature $\rF_\nabla$, such that $\nabla$ induces an associated basic complex \eqref{eq:Complexd7}
$$
B_\bullet:\quad 0\to\Omega^0_H(\fg_E) \xrightarrow[{  }]{ d_\nabla}\Omega^1_H(\fg_E) \xrightarrow[{  }]{d_7}\Omega_{6\oplus 1}(\fg_E)\to 0.
$$
Denote by $ \rH^k_B$ the cohomology of $B_\bullet$.  If    $\int_M\langle\cF\varphi,\varphi\rangle> 0 $ and $\int_M\langle\cR\varphi,\varphi\rangle> 0$ for all $0\neq\varphi\in\Omega^{2,0}(\fg_e)$, then  $\rH^2_B=0$,
where $\cF$ and $\cR$ are endomorphism on $\Omega^{2,0}(\fg_E)$ defined with aid of $(\cF\varphi)_{\mu\nu}$ and $(\cR\varphi)_{\mu\nu}$ in Proposition \ref{prop:bochner}
\end{theorem}
\begin{proof} 
Suppose that $h_2\neq 0$, i.e., there exist  $\varphi\in\Omega_{6\oplus 1}^2(\fg_E )$ such that $\varphi\in \Ker(d_7\circ d_7^\ast)$,   from    Lemma \ref{lem:SDcharacterization}, we know that $\varphi$ has the form $\varphi=\varphi^{2,0}+\ol{\varphi^{2,0}}+\omega\otimes \varphi^0$ for some $\varphi^{2,0}\in \Omega^{2,0}_H(\fg_E )$ and $\varphi^0\in\Omega_H^0(\fg_E )$. Since  $d_7$ is the composition of $d_\nabla$  with the projection on $\Omega^2_{6\oplus 1}(\fg_E)$, so the adjoin is essentially the composition of $d^{\ast}_\nabla$ with inclusion map $\Omega^2_{6\oplus 1}(\fg_E)\to\Omega^2(\fg_E )$,
hence 
\begin{align*}
0&=d_7d_7^\ast(\omega\otimes\varphi^0)
+d_7\left(\Big(\partial_\nabla+\ol{\partial}_\nabla)^\ast(\varphi^{2,0}+\ol{\varphi^{2,0}}\Big)\right) \\
&=d_7d_7^\ast(\omega\otimes\varphi^0)+d_7\left(\partial^{\ast}_\nabla\varphi^{2,0}+{\ol{\partial}_\nabla^{\ast}}\ol{\varphi^{2,0}} \right) 
\end{align*}    
and using the definition of $d_7$ \eqref{eq:def d7}  in the above equality, we obtain
\begin{align}\label{eq:prev main theorem 1}
0
&=\partial_\nabla\partial^{\ast}_\nabla\varphi^{2,0}
+\ol{\partial}_\nabla
\ol{\partial}_\nabla^{\ast}\ol{\varphi^{2,0}} 
+
\Big\langle\ol{\partial}_\nabla\partial^{\ast}_\nabla\varphi^{2,0}+
\ol{\partial}_\nabla \partial_\nabla^{\ast}\ol{\varphi^{2,0}},\omega\Big\rangle\omega
+d_7d_7^\ast(\omega\otimes\varphi^0) 
\end{align}
We note that $\partial_\nabla\varphi^{2,0}\in \Omega^{3,0}_H(\fg_E )=0$ and $\ol{\partial}_\nabla\ol{\varphi^{2,0}}\in \Omega^{0,3}_H(\fg_E )=0$, since $\Omega^3(\fg_E )=\Omega^3_H(\fg_E )\oplus \eta\wedge{\Omega^2_H(\fg_E )}$ and $\Omega^3_H(\fg_E )$ sits inside of the algebraic ideal generated by $\Omega^2_8$ (cf. \eqref{eq:Lk spaces} and \cite[Proposition~4.5]{portilla2023}), then \eqref{eq:prev main theorem 1}  is
\begin{equation}\label{eq:prev main theorem 2}
0 =\Delta_{\partial_\nabla}\varphi^{2,0} 
+ \Delta_{\ol{\partial}_\nabla}\ol{\varphi^{2,0}}
+\Big\langle\ol{\partial}_\nabla\partial^{\ast}_\nabla\varphi^{2,0}
+\ol{\partial}_\nabla \ol{\partial}_\nabla^{\ast}\ol{\varphi^{2,0}},\omega\Big\rangle\omega +d_7d_7^\ast(\omega\otimes\varphi^0) 
\end{equation} 
From Lemma \ref{lem:previo 2} the last two terms in \eqref{eq:prev main theorem 2} are of type $(1,1)$, hence by comparing the bidegree in \eqref{eq:prev main theorem 2},  we note that $\Delta_{\partial_\nabla}\varphi^{2,0}=0$ and   we can  apply the Weitzenböck formula \eqref{eq:weitzenbock}. We maintain the notation     $\varphi^{2,0}=\frac12\sum_{}\varphi_{ \tau\gamma}dz^\tau\wedge dz^\gamma$  and $\Delta_{\partial_\nabla}\varphi^{2,0}=\frac12\sum_{ }\phi_{\mu\nu}dz^\mu\wedge dz^\nu\in \Omega^{2,0}(\fg_E ),$ where
$
\phi_{\mu\nu} =
-\sum_{\alpha \beta }g^{\alpha\bar{\beta}}\widetilde{\nabla}_{\bar{\beta}}\widetilde{\nabla}_{\alpha}\varphi_{\mu\nu}-(\cF\varphi^{2,0})_{\mu\nu}-  (\cR\varphi^{2,0})_{\mu\nu}  
$   
\eqref{eq:weitzenbock}, 
we have 
\begin{align*}
4\left< \Delta_{\partial_\nabla}\varphi^{2,0},\varphi^{2,0}  \right> &=\sum_{\tau\gamma\mu\nu}\langle 
 \phi_{\mu\nu}dz^\mu\wedge dz^\nu,\varphi_{\tau\gamma}dz^\tau\wedge dz^\gamma\rangle \\ 
&=-\sum_{\tau\gamma\mu\nu\alpha\beta}g^{\mu\bar{\tau}}g^{\nu\bar{\gamma}}g^{\alpha\bar{\beta    }}B(\wt{\nabla}_{\bar{\beta}}\wt{\nabla}_{\alpha}\varphi_{\mu\nu}, \varphi_{\tau\gamma})\\
&\quad 
-\sum_{\tau\gamma\mu\nu}g^{\mu\bar{\tau}}g^{\nu\bar{\gamma}}
B((\cF\varphi^{2,0})_{\mu\nu}, \varphi_{\tau\gamma}) 
- \sum_{\tau\gamma\mu\nu} g^{\mu\bar{\tau}}g^{\nu\bar{\gamma}}
B((\cR\varphi^{2,0})_{\mu\nu}, \varphi_{\tau\gamma} ) 
\end{align*}
from \eqref{eq:prev main theorem 2} and the above computation 
\begin{align*}
0 &= 
\sum_{\tau\gamma\mu\nu\alpha\beta}g^{\mu\bar{\tau}}g^{\nu\bar{\gamma}}g^{\alpha\bar{\beta}}
B(\wt{\nabla}_{\bar{\beta}}\wt{\nabla}_{\alpha}\varphi_{\mu\nu}, \varphi_{\tau\gamma }) 
+ \Big\langle \cF(\varphi^{2,0}),\varphi^{2,0}\Big\rangle  
+ \Big\langle \cR(\varphi^{2,0} ),\varphi^{2,0}\Big\rangle. 
\end{align*}
Denote by  $\nabla^{1,0}$  be   the $(1,0)$-part o the connection $\nabla$, then we integrate the above equality and we use integration by parts and Stokes' Theorem on the first term to obtain   
\begin{equation}\label{eq:BoschnerMethod}
0=\left( \nabla^{1,0}\varphi^{2,0} ,\nabla^{1,0}\varphi^{2,0}\right) +  \int_M \Big\langle \cF(\varphi^{2,0}),\varphi^{2,0}\Big\rangle
+ \int_M \Big\langle\cR(\varphi^{2,0} )  ,\varphi^{2,0}\Big\rangle,  \end{equation}
then from \eqref{eq:BoschnerMethod}  we conclude  that $\varphi^{2,0}=0$ and $\nabla^{1,0}\varphi^{2,0}=0$.   
Consequently $\varphi$ is only given by  $\varphi=\omega\otimes \varphi^0$. Now,  using   Lemma \ref{lem:previo 2} we have
\begin{align*}
0
&=\Big\langle d_7\circ d_7^\ast(\varphi ), \omega\otimes \varphi^0\Big\rangle\\
&=\Big\langle d_7\circ d_7^\ast(\omega\otimes \varphi^0), \omega\otimes \varphi^0\Big\rangle\\ 
&=-4\ii \Big\langle \omega\otimes\langle \partial_\nabla\ol{\partial}_\nabla \varphi^0,\omega\rangle,\omega\otimes \varphi^0\Big\rangle\\
&= 8 \Big\langle\langle\partial_\nabla\bar{\partial}_\nabla\varphi^0, \omega \rangle , \varphi^0\Big\rangle\\
&=
8\Big\langle\sum_{\mu\nu}g^{\mu\bar{\nu}}\wt{\nabla}_\mu\wt{\nabla}_{\bar{\nu}}\varphi^0 , \varphi^0\Big\rangle
\end{align*}
by integrate  the above equality on $M$,  we obtain 
\begin{equation}\label{eq:varpi0parallel}
0=8\left( \nabla^{1,0} \varphi^0,\nabla^{1,0}\varphi^0\right)_M
\end{equation}
i.e., $\nabla \varphi^0=\nabla^{1,0} \varphi^0+\ol{\nabla^{1,0} \varphi^0}=0$. Since the space of sections $\Gamma(E)$ is trivial because $\nabla $ is a irreducible follows that   $\varphi^0=0$ and consequently  $h_2=0$. 
\end{proof}
\begin{remark}\label{rem:varphiMaintheorem} 
\begin{itemize}
\item 
Let $\varphi$ be of the form  $\varphi=\varphi^{2,0}+\ol{\varphi^{2,0}}+\omega\otimes \varphi^0\in \Ker(d_7\circ d_7^\ast)$. From the proof of Theorem \ref{thm:main theorem}, if $\varphi^{2,0}=0$ then $\varphi=0$. That is why in Section \ref{sec:positivity}, we only focus in demonstrating that $\varphi^{2,0}=0$.
\item We can have different situations in \eqref{eq:BoschnerMethod}, namely, for contact Calabi-Yau manifolds the term $\int_M\langle\cR(\varphi^{2,0} ),\varphi^{2,0}\rangle$ vanishes,  while for an Abelian  Lie group $\rG$,  $ \int_M\langle \cF(\varphi^{2,0}),\varphi^{2,0} \rangle=0.$
\end{itemize}
\end{remark}
\begin{proposition}\label{prop:Pararell}
Consider $\varphi\in \Ker(d_7\circ d_7^\ast)$ of the form  $\varphi=\varphi^{2,0}+\ol{\varphi^{2,0}}+\omega\otimes \varphi^0$ as in the proof of Theorem \ref{thm:main theorem}, then $ \varphi^{2,0}$ is parallel. Furthermore $\varphi$ is parallel, consequently all element in $H^2_B$ \eqref{eq:Complexd7} is zero or covariantly constant,
\end{proposition}
\begin{proof} 
From \eqref{eq:BoschnerMethod} we know that $\nabla^{1,0}\varphi^{2,0}=0$, in order to see that  $\varphi^{2,0}$ is parallel,  note that  $\delbar_\nabla  \varphi^{2,0}=0$, since it is a $(2,1)$-form and $\Omega^3_H(\fg_E )$ sits inside of the algebraic ideal generated by $\Omega^2_8$ \eqref{eq:Lk spaces}. We also know from \eqref{eq:prev main theorem 2} that  $\varphi^{2,0}$ is $\partial_\nabla$--harmonic. Furthermore we can prove that  $\varphi^{2,0}$ is  $d_\nabla$--harmonic. Using that  $\partial_\nabla\varphi^{2,0}=0=\partial_\nabla^\ast\varphi^{2,0}$ we compute 
\begin{align*}
d^{\ast}_\nabla d_\nabla \varphi^{2,0}&=d^{\ast}_\nabla (\del_\nabla +\delbar_\nabla )\varphi^{2,0}
=(\del_\nabla ^\ast+\delbar_\nabla ^\ast)\delbar_\nabla  \varphi^{2,0}\\
&=\del_\nabla ^\ast\delbar_\nabla  \varphi^{2,0}+\delbar_\nabla ^\ast\delbar_\nabla  \varphi^{2,0} 
=\Delta_{\delbar_\nabla}\varphi^{2,0}\\
&=0
\end{align*}
On the other hand, note that $d^\ast_\nabla\varphi^{2,0}=(\del_\nabla ^\ast+\delbar_\nabla ^\ast)\varphi^{2,0}=0$, hence $d_\nabla d_\nabla ^{\ast}\varphi^{2,0}=0$  
and  consequently we obtain 
$$
0=\Delta_{d_\nabla }\varphi^{2,0}=\Delta_{\delbar_\nabla }\varphi^{2,0}
$$  
Now, to see that  $\varphi$ is parallel it is enough to proof that $d_7\circ d_7^\ast(\omega\otimes\varphi^0)=0$ and the result follows from the same argument to obtain \eqref{eq:varpi0parallel}. 
In principle $d_7\circ d_7^\ast(\varphi^{2,0}+\ol{\varphi^{2,0}} )\in \Omega^2_{6\oplus 1}(\fg_E)$, however since  $d_7\circ d_7^\ast(\omega\otimes\varphi^0)$ is type $(1,1)$ [cf. Lemma \ref{lem:previo 2}] we note that  
\begin{align*}
0&= 
\langle\varphi^{2,0}+\ol{\varphi^{2,0}}  ,  d_7\circ d_7^\ast(\omega\otimes\varphi^0) \rangle
=
\langle d_7\circ d_7^\ast(\varphi^{2,0}+\ol{\varphi^{2,0}} ), \omega\otimes\varphi^0 \rangle 
\end{align*} 
follows that $d_7\circ d_7^\ast(\varphi^{2,0}+\ol{\varphi^{2,0}} )\in \Omega^2_{6}(\fg_E)$. 
Note that 
$$ 
0  = d_7\circ d_7^\ast(\varphi) 
   = \underbrace{d_7\circ d_7^\ast(\varphi^{2,0}+\ol{\varphi^{2,0}} )}_{\Omega^2_{6}(\fg_E)} 
   + \underbrace{d_7\circ d_7^\ast(\omega\otimes\varphi^0)}_{\Omega^2_{1}(\fg_E)}\\ 
$$  
follows that   both  $\omega\otimes\varphi^0 $ and  $\varphi^{2,0}+\ol{\varphi^{2,0}}$ belong to  $ \ker(d_7\circ d_7^\ast)$.  
\end{proof} 
\section{Positivity conditions}
\label{sec:positivity}
The main argument to obtain the vanishing Theorem \ref{thm:main theorem} is to assume that  $\int_M\langle \cF(\varphi^{2,0}),\varphi^{2,0}\rangle> 0$ and $\int_M\langle \cR(\varphi^{2,0}),\varphi^{2,0}\rangle> 0$  for al $\varphi^{2,0}\neq 0$ in \eqref{eq:BoschnerMethod}. 
In  contrast with an analogous result in  Kähler surfaces \cite{itoh1983moduli} and its foliated version in $5$-dimensions \cite{Baraglia2016}, in our case, the function  $\langle \cF(\varphi^{2,0}),\varphi^{2,0}\rangle$ does not vanishes automatically. We can provide conditions for positivity of  $\langle \cR(\varphi^{2,0}),\varphi^{2,0}\rangle$.\  
We adopt the following notation, in a local chart  $(U,z_1,z_2,z_3,\xi)$, where $z_j=x_j-\ii\Phi x_j$, we  denote by $\{\frac{\partial}{ \partial z_i }\}_{i=1}^3 $ and $\{\frac{\partial}{\partial\zbar_i}\}_{i=1}^3 $ the basis for     $\rH^{1,0}$ and $\rH^{0,1}$  respectively  [cf. Appendix \ref{sec:tKahlerGeometry}], to simplify the notation, eventually we denote $ \frac{\partial}{ \partial z_i }=\partial_{z_i}$ and $\frac{\partial}{\partial\zbar_i}=\partial_{\zbar_i}$ and so on. 
\subsection{Positivity of  \texorpdfstring{$ \langle\cR\varphi^{2,0},\varphi^{2,0}\rangle$}{lg} 
}\label{sec:Ricc}
We analyze the function  $\langle\cR\varphi^{2,0},\varphi^{2,0}\rangle$ 
and we deduce  conditions under which  positive (cf. Proposition \ref{prop:Ricci} at the end of this section). In a suitable normal  frame such that $g_{\alpha\bar{\beta}}=\delta_{\alpha\beta}$ in a fixed point $x\in M$, the component $(\cR\varphi^{2,0})_{\mu\nu}$ in the Weitzenböck formula \eqref{eq:weitzenbock} are given by 
\begin{equation}\label{eq:cR2}
(\cR\varphi^{2,0})_{\mu\nu}
=\sum_{\alpha}  \left( \rR_{\bar{\alpha}\mu } \varphi_{\alpha\nu}-\rR_{\bar{\alpha}\nu } \varphi_{\alpha\mu}\right)     
\end{equation}  
In a orthonormal frame , the transversal Ricci operator  is defined by [cf. Definition \ref{def:RicciT}]
$$ 
\Ric^T(X)=\sum_{\mu}R(X,\partial_{\bar{\mu}})\partial_\mu. 
$$   
We extend it to a $2$--form by 
$ 
(\Ric^T\wedge \rI)(X,Y)=\Ric^T(X)\wedge Y+X\wedge  \Ric^T(Y), 
$ 
where $I$ is the identity operator.  Let $\varphi\in \Omega^{2,0}(\fg_E)$, hence 
\begin{align*}
(\cR\varphi^{2,0})( \partial_\mu,\partial_\nu )
&=
\sum_{\alpha}  \Big( (Ric^T)_{\mu\bar{\alpha}}\varphi^{2,0}(\partial_\alpha,\partial_\nu)+(Ric^T)_{\alpha\nu} \varphi^{2,0}(\partial_\mu,\partial_\alpha)\Big)\\
&=   
\varphi^{2,0} \left(\sum_{\alpha}\rR( \partial_\mu,\partial_{\bar{\alpha}} )\partial_\alpha,\partial_\nu\right)
+\varphi^{2,0}\left(\partial_\mu,\sum_{\alpha}\rR(\partial_{\bar{\alpha}},\partial_\nu)\partial_\alpha\right) \\
&=
\varphi^{2,0} \left(\Ric^T(\del_\mu),\partial_\nu\right)+ \varphi^{2,0}\left(\partial_\mu,\Ric^T(\del_\nu)\right) \\
&= 
\varphi^{2,0}( \Ric^T( \partial_\mu)\wedge\partial_\nu +  \partial_\mu\wedge \Ric^T( \partial_\nu))\\ 
&=
\left(\varphi^{2,0}\circ \Ric^T \wedge I\right) (\partial_\mu, \partial_\nu)   
\end{align*}  
From the above computation  we obtain  an expression for the action $\cR$ as follow
\begin{equation}\label{eq:RRiccI} 
\cR\varphi^{2,0}=\varphi^{2,0}\circ (\Ric^T\wedge \rI) 
\end{equation}
Note in particular that if $M$ is Sasaki-Einstein with constant $\lambda$, $Ric^T\wedge I=2\lambda I$.  The following follows 
\begin{proposition}\label{prop:Ricci}
Let $E\to M$ be a complex Sasakian vector bundle on a compact, connected  Sasakian manifold $(M,\cS)$  with  positive transverse Ricci curvature, and let $\nabla\in\cA(E)$ be an irreducible  SDCI such that $\int_M\langle\cF\varphi,\varphi\rangle\geq 0$ for all $\varphi\in \Omega^{2,0}(\fg_E)$. Then    $\rH^2_B=0$. [cf. Theorem \ref{thm:main theorem}].   
\end{proposition}
In the particular case of \textit{Sasaki-Einstein manifolds} (Proposition \ref{prop:sasaki eisntein}), we can argue as follow. The term in \eqref{eq:cR2}  is skew-symmetric in $\mu$ and $\nu$, hence, considering $\mu<\nu\in \{1,2,3\}$,   we obtain
\begin{align*}
 \cR_{12}(\varphi^{2,0})&=
 (\rR_{\bar{1}1 } \varphi_{1 2} - \rR_{\bar{1}2 } \varphi_{1 1} )
+(\rR_{\bar{2}1 } \varphi_{2 2} - \rR_{\bar{2}2 } \varphi_{2 1} )
+(\rR_{\bar{3}1 } \varphi_{3 2} - \rR_{\bar{3}2 } \varphi_{3 1} )\\ 
&=
 \rR_{\bar{1}1 } \varphi_{1 2} 
+\rR_{\bar{2}2 } \varphi_{1 2}  
+\rR_{\bar{3}1 } \varphi_{3 2} 
+\rR_{\bar{3}2 } \varphi_{1 3},\\ 
\end{align*}
we can do analogous computations to obtain 
\begin{equation}\label{eq:bochner theta} 
\begin{array}{l}
(\cR\varphi^{2,0})_{13}=(\rR_{\bar{1}1}+\rR_{\bar{3}3}) \varphi_{13}+\rR_{\bar{2}1 } \varphi_{2 3}+ \rR_{\bar{2}3 } \varphi_{12}  \\[4pt]
(\cR\varphi^{2,0})_{23}=(\rR_{\bar{1}1}+\rR_{\bar{3}3})\varphi_{23}
+ \rR_{\bar{1}2 }\varphi_{1 3}-\rR_{\bar{1}3 } \varphi_{1 2}  \end{array}
\end{equation} 
Note that in a point $x\in M$ we may assume that the Ricci tensor is diagonalized (cf. \cite[p.~227]{zheng2000complex}).
Hence, from \eqref{eq:bochner theta}  follows that   
If $M$ is Sasaki-Einstein\footnote{
In this case $\Ric=6g$ and $\Ric^T=8g^T$  cf. Definition \ref{def:sasaki einstein}} manifold $\rR_{\bar{i}i}>0$ hence Proposition \ref{prop:Ricci} follows, but in this case   we can obtain that $\langle \cR\varphi^{2,0},\varphi^{2,0}\rangle$ is positive definite  in a simpler way 
\begin{align*} 
\langle \cR\varphi^{2,0},\varphi^{2,0}\rangle
&=
(\rR_{\bar{1}1}+\rR_{\bar{2}2})\Vert\varphi_{12}\Vert^2+(\rR_{\bar{1}1}+\rR_{\bar{3}3})\Vert\varphi_{13}\Vert^2+(\rR_{\bar{2}2}+\rR_{\bar{3}3})\Vert\varphi_{23}\Vert^2\\
&> 0.\\
\end{align*}
\subsection{Pointwise description of  \texorpdfstring{$\langle\cF\varphi^{2,0},\varphi^{2,0}\rangle$}{lg}}\label{sec:positivityF} 
We maintain the notation for   $\varphi^{2,0}\in \Omega^{2,0}(\fg_E)$ with the same hypothesis as in the proof of the Theorem \ref{thm:main theorem}.
We want write down an expression for $\langle \cF \varphi^{2,0},\varphi^{2,0}\rangle, $ fix a point $x\in M$ and  normal transverse coordinates such that $g^{\alpha\bar{\beta}}=\delta^\alpha_\beta$, hence 
$$  
(\cF\varphi^{2,0})_{\mu\nu} = \sum_{\alpha }\left([\varphi_{\alpha\nu},\rF_{\mu\bar{\alpha}}]
-[\varphi_{\alpha\mu},\rF_{\nu\bar{\alpha}}]\right). 
$$  
Since $(\cF\varphi^{2,0})_{\mu\nu}$ is anti-symmetric in $\mu,\nu$  on a $7$-dimensional manifold we consider $\mu<\nu\in \{1,2,3\}$,   we obtain 
\begin{align*}
(\cF\varphi^{2,0})_{12}
  &= [\varphi_{1 2},\rF_{1\bar{1}}]-[\varphi_{1 1},\rF_{2\bar{1}}] 
   + [\varphi_{2 2},\rF_{1\bar{2}}]-[\varphi_{2 1},\rF_{2\bar{2}}] 
   + [\varphi_{3 2},\rF_{1\bar{3}}]-[\varphi_{3 1},\rF_{2\bar{3}}] \\
 &= [\varphi_{1 2},\rF_{1\bar{1}}] +[\varphi_{12},\rF_{2\bar{2}}] 
    + [\varphi_{3 2},\rF_{1\bar{3}}]-[\varphi_{3 1},\rF_{2\bar{3}}] \\
 &= [\varphi_{13},\rF_{2\bar{3}}] -[\varphi_{23},\rF_{1\bar{3}}]-[\varphi_{12},\rF_{3\bar{3}}],
\end{align*}
where we used that $\rF_{1\bar{1}} +\rF_{2\bar{2}}+\rF_{3\bar{3}}=0$ [cf. Proposition \ref{prop: selfdual = HYM}],  analogous computations for $\cF_{13}$ and $\cF_{23}$ provide 

\begin{equation}\label{eq:bochner Xi}
\begin{array}{l}
(\cF\varphi^{2,0})_{12}=[\varphi_{32},\rF_{1\bar{3}}] + [\varphi_{13},\rF_{2\bar{3}}] + [\varphi_{21},\rF_{3\bar{3}}]\\[4pt]
(\cF\varphi^{2,0})_{13}=[\varphi_{23},\rF_{1\bar{2}}] + [\varphi_{31},\rF_{2\bar{2}}] + [\varphi_{12},\rF_{3\bar{2}}] \\[4pt]
(\cF\varphi^{2,0})_{23}=[\varphi_{32},\rF_{1\bar{1}}] + [\varphi_{13},\rF_{2\bar{1}}] + [\varphi_{21},\rF_{3\bar{1}}]  
\end{array}
\end{equation}

We recall the ad-invariance of the metric in $\fg$, i.e., for all $a,b,c\in\fg$,  
$ 
B([a,b],c)=-B(a,[b,c]),
$  
then we have   
\begin{align*}
\left< \cF(\varphi^{2,0}),\varphi^{2,0}\right>=& 
\sum_{\mu\nu}B\Big(\sum_{\alpha}
     ([\varphi_{\alpha\nu},\rF_{\mu\bar{\alpha}}]
     -[\varphi_{\alpha\mu},\rF_{\nu\bar{\alpha}}]),\varphi_{\mu\nu}\Big)\\ 
=&-\sum_{\mu\nu \alpha}
 B\Big([\varphi_{\alpha\nu},\varphi_{\mu\nu}],\rF_{\mu\bar{\alpha}}\Big)
-B\Big([\varphi_{\alpha\mu},\varphi_{\mu\nu}],\rF_{\nu\bar{\alpha}}\Big)   
\end{align*}
Hence, for the non zero components of $\varphi^{2,0}$ we have
\begin{align*}
-\left< \cF(\varphi),\varphi\right >=& 
  B([\varphi_{12},\varphi_{12}],\rF_{1\bar{1}}) 
- B([\varphi_{21},\varphi_{12}],\rF_{2\bar{2}})
+ B([\varphi_{32},\varphi_{12}],\rF_{1\bar{3}}) 
- B([\varphi_{31},\varphi_{12}],\rF_{2\bar{3}})\\
+& 
  B([\varphi_{13},\varphi_{13}],\rF_{1\bar{1}}) 
+ B([\varphi_{23},\varphi_{13}],\rF_{1\bar{2}}) 
- B([\varphi_{21},\varphi_{13}],\rF_{3\bar{2}})
- B([\varphi_{31},\varphi_{13}],\rF_{3\bar{3}})\\
+& 
  B([\varphi_{13},\varphi_{23}],\rF_{2\bar{1}}) 
- B([\varphi_{12},\varphi_{23}],\rF_{3\bar{1}})
+ B([\varphi_{23},\varphi_{23}],\rF_{2\bar{2}}) 
- B([\varphi_{32},\varphi_{23}],\rF_{3\bar{3}})\\
\end{align*}
by reducing the last expression  we obtain 
\begin{align*}
\left< \cF(\varphi),\varphi\right>&= 
   B([\varphi_{13},\varphi_{23}],\rF_{1\bar{2}})
  -B([\varphi_{12},\varphi_{23}],\rF_{1\bar{3}}) 
  -B([\varphi_{13},\varphi_{12}],\rF_{2\bar{3}})\\
& \;
  -B([\varphi_{13},\varphi_{23}],\rF_{2\bar{1}}) 
  +B([\varphi_{12},\varphi_{23}],\rF_{3\bar{1}})
  +B([\varphi_{13},\varphi_{12}],\rF_{3\bar{2}})\\
&= 
   B([\varphi_{13},\varphi_{23}],\rF_{1\bar{2}}-\rF_{2\bar{1}})
  +B([\varphi_{12},\varphi_{23}],\rF_{3\bar{1}}-\rF_{1\bar{3}}) 
  +B([\varphi_{13},\varphi_{12}],\rF_{3\bar{2}}-\rF_{2\bar{3}})\\
&= 
   B([\varphi_{13},\varphi_{23}],\rF_{1\bar{2}}+\ol{\rF_{1\bar{2}}})
  +B([\varphi_{12},\varphi_{23}],\rF_{3\bar{1}}+\ol{\rF_{3\bar{1}}}) 
  +B([\varphi_{13},\varphi_{12}],\rF_{3\bar{2}}+\ol{\rF_{3\bar{2}}})  
\end{align*}    
From Proposition \ref{prop: selfdual = HYM} that $\ol{\rF_\nabla}=\rF_\nabla$, hence 
$\rF_{\mu\bar{\nu}}=-\ol{\rF_{\nu\bar{\mu}}}$, consequently
\begin{equation}\label{eq:I}
\langle \cF(\varphi),\varphi\rangle=
2\Big\{\langle [\varphi_{13}, \varphi_{23}],\Re(\rF_{1\bar{2}})\rangle
      +\langle [\varphi_{12}, \varphi_{23}],\Re(\rF_{3\bar{1}})\rangle 
      +\langle [\varphi_{12}, \varphi_{13}],\Re(\rF_{2\bar{3}})\rangle\Big\}
\end{equation}
Now  we use the fact that $F_\nabla\in\Omega_8^2(\fg_E)$ is a SDCI and  $\varphi^{2,0}$  is the $(2,0)$-part of $\varphi\in\Omega_{6\oplus 1}^2(\fg_E)$.   We recall that $\Omega_8^2$ is locally generated   cf. \cite[eq.~(2.17) and (2.19)]{portilla2023}
\begin{equation}\label{eq:wi}
\begin{array}{*{2}c}
w_1 = \frac{1}{2}(dz^1\wedge d\zbar^2-dz^2\wedge d\zbar^1) , & w_2 = \frac{\ii}{2}(dz^1\wedge d\zbar^2+dz^2\wedge d\zbar^1), \\[3pt]
w_3 = \frac{1}{2}(dz^1\wedge d\zbar^3-dz^3\wedge d\zbar^1),  & w_4 = \frac{\ii}{2}(dz^1\wedge d\zbar^3+dz^3\wedge d\zbar^1), \\[3pt]
w_5 = \frac{1}{2}(dz^2\wedge d\zbar^3-dz^3\wedge d\zbar^2) , & w_6 = \frac{\ii}{2}(dz^2\wedge d\zbar^3+dz^3\wedge d\zbar^2)  \\[3pt] 
w_7 = \frac{\ii}{2}(dz^1\wedge d\zbar^1-dz^3\wedge d\zbar^3),& w_8 = \frac{\ii}{2}(dz^2\wedge d\zbar^2-dz^3\wedge d\zbar^3).  \end{array} 
\end{equation} 
 and  $\Omega^2_6$ is locally generated by   
\begin{equation}
\begin{array}{*{2}c}\label{eq:vi}
v_1 = \frac{1}{2}(dz^1\wedge dz^2+d\zbar^1\wedge d\zbar^2) , & v_2 = \frac{\ii}{2}(d\zbar^1\wedge d\zbar^2-dz^1\wedge dz^2),  \\[3pt]
v_3 =\frac{1}{2}(dz^1\wedge dz^3+d\zbar^1\wedge d\zbar^3),   & v_4 = \frac{\ii}{2}(d\zbar^1\wedge d\zbar^3-dz^1\wedge dz^3),  \\[3pt]
v_5 =  \frac{1}{2}(dz^2\wedge dz^3+d\zbar^2\wedge d\zbar^3), & v_6 = \frac{\ii}{2}(d\zbar^2\wedge d\zbar^3-dz^2\wedge dz^3).
\end{array}
\end{equation} 
We can locally write $F_\nabla=\sum_{i=1}^8 w_i\otimes a_i$ for some $a_i\in \su(r)$. Since $\sum_\mu \rF_{\mu\bar{\mu}}=0,$ 
and $\rF_{\mu\bar{\nu}}=-\ol{\rF_{\nu\bar{\mu}}},$ from \eqref{eq:wi} we have
$$
\begin{array}{lll}
\rF_{1\bar{1}}=\frac{\ii }{2}a_7, &\rF_{2\bar{2}}=\frac{\ii }{2}a_8, & \rF_{3\bar{3}}=-\frac{\ii }{2}a_7-\frac{\ii }{2}a_8\\ [4pt]
\rF_{1\bar{2}}=\frac{a_1}{2}+\frac{\ii a_2}{2}, & \rF_{1\bar{3}}=\frac{a_3}{2}+\frac{\ii a_4}{2} & \rF_{2\bar{3}}=\frac{a_5}{2}+\frac{\ii a_6}{2},    
\end{array}
$$
according to this  \eqref{eq:I} is  
$ 
 \left< \cF(\varphi),\varphi\right>=
\langle[\varphi_{13},\varphi_{23}], a_1 \rangle
+ \langle[\varphi_{12},\varphi_{23}],-a_3 \rangle 
+ \langle[\varphi_{12},\varphi_{13}], a_5 \rangle.   
$\\

On the other hand, denoting by  $\varphi=\sum_i v_i\otimes b_i$ for some $b_i\in \su(r)$ \eqref{eq:vi}, the  $(2,0)$-part is 
\begin{align*} 
\varphi^{2,0}&=
        \left(\frac{b_1}{2}-\frac{\ii b_2}{2}\right)dz^1\wedge dz^2
+       \left(\frac{b_3}{2}-\frac{\ii b_4}{2}\right)dz^1\wedge dz^3
+       \left(\frac{b_5}{2}-\frac{\ii b_6}{2}\right)dz^2\wedge dz^3 
\end{align*}
Consequently, 
\begin{equation}\label{eq:2F}
2\Re\langle \cF(\varphi^{2,0}),\varphi^{2,0}\rangle 
=  \langle [b_3,b_5]-[b_4,b_6],a_1\rangle 
-  \langle [b_1,b_5]-[b_2,b_6],a_3\rangle   
+  \langle [b_1,b_3]-[b_2,b_4],a_5\rangle      
\end{equation}

Accordingly to \eqref{eq:2F}, for  U$(1)$-bundles it is clear that $\cF$ is positive, consider the following example, this  example of Sasakian structure is  constructed by the first author and A. Moreno  in \cite{andres}.\footnote{For details on the construction of the Sasakian structure for this example, we refer to \cite[Section  2.2]{andres}.} 
\begin{example}\label{ex:stiefel}
We consider a homogeneous  principal $\SO(3)$-bundle $P_1=\SO(5)\times_{\SO(3)}\SO(3)\to V^{5,2}$ on  the Stiefel manifold $V^{5,2}=\SO(5)/\SO(3)$. Denote by $E_{ij}$  the $5\times 5$ matrix with $1$ at the $(ij)$-entry and $0$ elsewhere  and a fixed basis of $\so(5)$  
\begin{align}\label{eq: so(5)_basis}
\begin{split}
e_1=E_{12}-E_{21},&\quad e_2= E_{13}-E_{31}, \quad e_5=E_{23}-E_{32}, \quad e_8=E_{34}-E_{43}, \\
                  &\quad e_3=E_{14}-E_{41},  \quad e_6=E_{24}-E_{42}, \quad e_9=E_{35}-E_{53},  \\
                  &\quad e_4=E_{15}-E_{51},  \quad e_7=E_{25}-E_{52}, \quad e_{10}=E_{54}-E_{45}, 
\end{split}
\end{align} 
Denote by $\fm_1:=\Span\{e_{1}\},$ $\fm_2:=\Span\{ e_2,e_3,e_4\}$, and $\fm_3:=\Span\{ e_5,e_6,e_7\}$ . The subalgebra $\so(2)\oplus\so(3)=\langle e_1,e_8,e_9,e_{10}\rangle \subset\so(5)$ provides a $\ad(\so(3))$-invariant reductive decomposition $\so(5)=\so(2)\oplus\so(3)\oplus \fm_2 \oplus\fm_3$ such that $V^{5,2}\to \SO(5)/(\SO(2)\times\SO(3))$ is a circle fibration. For certain conditions on $\{y_1, y_2 ,y_3\}\subset \R^+$  \cite{andres}, the  $\SO(5)$-invariant Sasakian metric  on $V^{2,0}$ are given by  
$$
 g=y_1B\vert_{\fm_1}+y_2B\vert_{\fm_2}+y_3B\vert_{\fm_3}.
$$ 
Consider the homogeneous  principal $\SO(3)$-bundle  $P_1=\SO(5)\times_{\SO(3)}\SO(3)\to V^{5,2}$ induced by  the   identity groups homomorphism \cite[Lemma~1.4.5]{cap2009parabolic}.\\

The linear map $\alpha\colon\so(5)\rightarrow \so(3)$  given by 
$\alpha=e^8\otimes e_8+e^9\otimes e_9+e^{10}\otimes e_{10}\in \cA(P_1)$  is an invariant SO$(3)$-connection with curvature 
$ 
F_\alpha:=\frac{1}{y_2}(w_1\otimes e_8+ w_3\otimes e_9+w_5\otimes e_{10}),
$ 
where  $y_2>0.$  For $y_1=\frac{9}{16}$, $y_2=y_3=\frac{3}{8}$ the metric $g$ is Sasaki-Einstein $\alpha$  is a SDCI.  
From \eqref{eq:2F} we have,
\begin{equation}\label{eq:2F1}
2\Re\langle \cF(\varphi^{2,0}),\varphi^{2,0}\rangle 
 =  \langle [b_3,b_5]+[b_6,b_4],e_8\rangle 
+  \langle [b_5,b_1]+[b_2,b_6],e_9\rangle   
+  \langle [b_1,b_3]+[b_4,b_2],e_{10}\rangle.    
\end{equation}  
Therefore according to our convention, we can suppose that $b_j=a_{ji}e_i$  where $a_{ij}\in \R$, $j=1,\cdots,6$ and $e_i\in\so(3)$. since $[e_i,e_j]
=\epsilon_{ijk}e_k$ for $i,j,k\in \{8,9,10\}$ \eqref{eq: so(5)_basis}, we have that 
\begin{align*}
 [b_j, b_k]&=[a_{ji}e_i, a_{kl}e_l]= a_{ji}a_{kl}[e_i,e_l]
 =a_{ji}a_{kl}\epsilon_{ilt}e_t.
\end{align*}
Hence, for example  
$$
[b_3,b_5]+[b_6,b_4] =
a_{3i}a_{5l}\epsilon_{ilt}e_t+a_{6i}a_{4l}\epsilon_{ilt}e_t 
$$  
In the above equality  we set $t=8$ to obtain the non zero value in the first term of \eqref{eq:2F1}, we obtain 
$$
(a_{3i}a_{5l}\epsilon_{il8}+a_{6i}a_{4l}\epsilon_{il8})e_8=
(a_{39}a_{5\;10 }-a_{3\;10}a_{59}+a_{69}a_{4\;10}-a_{6\;10}a_{49})e_8
$$
we can make  analogous computations for $ [b_5,b_1]+[b_2,b_6]$ and $ [b_1,b_3]+[b_4,b_2]$ however is clear that $\Re\langle \cF(\varphi^{2,0}),\varphi^{2,0}\rangle$ can assume  both  positive and negative sign for appropriate choices of $a_{ij}\in \R$. 
\end{example} 
\subsubsection{An Estimate}\label{sec:estimate}
In this Section, we maintain the notation and the hypothesis as in the proof of  Theorem \ref{thm:main theorem}. Let $\varphi^{2,0}\in \Omega^{2,0}(\fg_E)$ and denote by $\lambda$ the minimal eigenvalue of the operator $\Ric^T\wedge I$ \eqref{eq:RRiccI}. From \eqref{eq:BoschnerMethod}  we have
\begin{equation}\label{eq:Bochner2}
(  \nabla^{1,0}\varphi^{2,0},\nabla^{1,0}\varphi^{2,0} ) 
\leq -\int_M \Big\{\langle \cF(\varphi^{2,0}),\varphi^{2,0}\rangle
+\vert \lambda\vert  \Vert \varphi^{2,0}\Vert^2 \Big\} 
\end{equation}
Note that if the subintegral quantity in the right hand side of \eqref{eq:Bochner2} is positive,  this produce a contradiction  which indicates that $\varphi^{2,0}$ must be zero. Now, since both $\langle\cF(\varphi^{2,0}),\varphi^{2,0}\rangle
$ and $ \Vert \varphi^{2,0}\Vert^2 $ are homogeneous quadratic in the variable $\varphi^{2,0}$, hence we look for conditions in which $\vert  \langle\cF(\varphi^{2,0}),\varphi^{2,0}\rangle\vert< \vert \lambda\vert  \Vert \varphi^{2,0}\Vert^2$, because, in this case  the positive term $\vert \lambda\vert  \Vert \varphi^{2,0}\Vert^2$ dominates the right hand side of \eqref{eq:Bochner2}.\\

For $1\leq \mu,\nu\leq 3$, we define the symmetric, trace  free matrix $A=(a_{\mu\nu})$ where $a_{\mu\nu}=\Vert\varphi_{\mu\nu}\Vert$, analogously define $B=(b_{\mu\nu})$ where  $b_{\mu\nu}=\Vert\rF_{\mu\bar{\nu}}\Vert.$
Denoting by $\varphi^{2,0}=\frac12\sum\varphi_{\mu\nu}dz^\mu\wedge dz^\nu$ and $\cF\varphi^{2,0}=\frac12\sum F_{\mu\bar{\nu}}dz^\mu\wedge d\zbar^\nu$ we have that $2\langle\cF(\varphi^{2,0}),\varphi^{2,0}\rangle=\sum_{\mu\nu \alpha}
\langle[\varphi_{\alpha\nu},\varphi_{\nu\mu}],\rF_{\mu\bar{\alpha}}\rangle $ and we note that 
\begin{align*}
\Big\vert\sum_{\mu\nu \alpha}
\langle[\varphi_{\alpha\nu},\varphi_{\nu\mu}],\rF_{\mu\bar{\alpha}}\rangle\Big\vert
&\leq 
\sum_{\mu\nu\alpha}
\vert\langle[\varphi_{\alpha\nu},\varphi_{\nu\mu}],\rF_{\mu\bar{\alpha}}\rangle \vert    
\leq 
\sum_{\mu\nu\alpha}\Vert  [\varphi_{\alpha\nu},\varphi_{\nu\mu}]\Vert \Vert \rF_{\mu\bar{\alpha}}\Vert  \\
&\leq 
\sqrt{2}\sum_{\mu\nu\alpha}\Vert\varphi_{\alpha\nu}\Vert\Vert \varphi_{\nu\mu}]\Vert \Vert\rF_{\mu\bar{\alpha}}\Vert
=\sqrt{2}\Tr(A^2B)\\
&\leq 
\sqrt{2}\Vert F^\nabla\Vert \Vert\varphi\Vert^2. 
\end{align*} 
The last line is justified as follow, first we used that $\Vert [A,B] \Vert \leq\sqrt{2}\Vert A\Vert\Vert B\Vert$ (cf. \cite[Lema 2.30]{bourguignon1981stability}). And on  other hand note that  since $\Vert\varphi\Vert^2=\sum_{\mu\nu}\Vert \varphi_{\mu\nu}\Vert^2$ define a norm for linear maps $\varphi\colon\Lambda^{2,0}\to \fg_E$, hence we have that 
\begin{align*}
\Vert A\Vert^2 
&=\Tr(A^2)
=\sum_{\mu\nu}a_{\mu\nu}a_{\nu\mu}
=\sum_{\mu\nu}a_{\mu\nu}^{2}
=\sum_{\mu\nu}\Vert\varphi_{\mu\nu}\Vert^{2} 
=  \Vert\varphi\Vert^2    
\end{align*} 
and analogously for $\Vert B\Vert^2$.  Consequently, 
$ 
\tr(A^2B)  \leq \Vert A^2\Vert\Vert B\Vert 
\leq\Vert A\Vert^2\Vert B\Vert= \Vert\varphi\Vert^2\Vert B\Vert= \Vert\varphi  \Vert^2\Vert F^\nabla\Vert. 
$ 
We summarize the discussion in this section in the   following 
\begin{proposition}\label{prop:energyF}
Let $(M,\cS)$ be a compact, connected Sasakian manifold [cf. Definition \ref{thm:app sasakian manifold}] with $\Ric^T>0$, and assume that the minimal eigenvalue of the operator $\Ric^T\wedge I$ is positive \eqref{eq:RRiccI}.  
let $E\to M$ be a complex Sasakian Vector bundle \cite{Boyer2008} and   $\nabla\in\cA(E)$  an irreducible  SDCI with curvature $2$-form $\rF_\nabla$. If the curvature $2$--form satisfy $\Vert F^\nabla\Vert< \sqrt{2} \lambda $ then  $h^2=0$ \eqref{eq:Complexd7}.    
\end{proposition}
\begin{proof}
If there exists    $\varphi=\varphi^{2,0}+\ol{\varphi^{2,0}}+\varphi^0\otimes \omega $ non zero section in  the same hypothesis as in the proof of Theorem \ref{thm:main theorem}, then  the previous discussion shows that  from \eqref{eq:Bochner2}   the  $(2,0)$--part $\varphi^{2,0}$ must be zero.  From this we can argue in the same way as in the  proof of  [cf. Theorem \ref{thm:main theorem}] \eqref{eq:varpi0parallel}.
\end{proof}
\subsubsection{The Yang-Mills condition and stability of a SDCI}
On the space $\cA(E)$  of $\rG$--connections it is defined the Yang-Mills  (YM) functional by 
\begin{equation}\label{eq:YMfunctional}
\cY\cM(\nabla):=\int_M\Vert F_\nabla\Vert^2,
\end{equation}
critical points of the YM-functional are called \textit{Yang-Mills (YM) connections}. In fact, $\nabla$ is a YM-connections if and only if 
$ 
d_\nabla^\ast F_\nabla=0.
$  
In general, from Bianchi identity, the contact instantons equation \eqref{eq: instanton} satisfy the YM--equation with torsion
\begin{equation}\label{eq:YMtorsion}
d_\nabla(\ast F_\nabla)=d\eta^2\wedge F_A.    
\end{equation} 
From Lemma \ref{lem:SDcharacterization} item (i), we know that   a SDCI satisfy 
$\langle F_A,d\eta\rangle=0$, 
follows that  
$ 
F_A\wedge\ast d\eta=\langle F_A,d\eta\rangle\vol= F_A\wedge\eta\wedge d\eta^2=0,
$
consequently  the torsion term is $F_A\wedge d\eta^2=0$  since  $\eta$ is nowhere vanishing, hence from \eqref{eq:YMtorsion} follows that  SDCI are YM-connections (the same argument holds for ASDCI).\\

Of particular interest are YM-connections which minimize \eqref{eq:YMfunctional} locally, at such a connections the second variation $ 
\dtzer\cY\cM(\nabla_t) 
$ 
is non-negative for any smooth family of connections $\nabla_t$, $\vert t\vert<\epsilon$ and $\nabla_0=\nabla$. A minimizing YM-connection is called \textbf{weak stable}.\ 
Let $\nabla=\nabla_0$ a YM-connection, denote by $B=\dtzero \nabla_t$  and define $\cR\colon \Omega^1(\fg_E)\to \Omega^1(\fg_E)$ given in an orthonormal frame $(X_1,\cdots X_n)$ by
$$
(\cR\alpha)(X)=\sum_{j=1}^n[F_{X_j,X},\alpha(X)].
$$
For variations $B\in \ker(d^\ast_\nabla)$ the second variation is given by \footnote{For a detailed study on stability of Yang-Mills connection the reader can consult \cite{bourguignon1981stability}}  
$ 
\dtzer\cY\cM(\nabla_t)=\int_M  \langle \mathscr{S}(B),B \rangle
$ 
(cf. \cite[Theorem 6.8]{bourguignon1981stability}) where 
\begin{equation}\label{eq:Secondvariation}
 \mathscr{S}(B)=\nabla^\ast\nabla B+B\circ \Ric+2\cR(B).  
\end{equation}  
We know that $\nabla^\ast\nabla$ is a positive operator, suppose furthermore that the Ricci tensor satisfy $\Ric\geq c I$ for some $c>0$, i.e., for all $X\in TM$, $\Ric(X,X)\geq c \Vert X \Vert^2$. Let 
$\{e_i\}$ be a orthonormal frame which diagonalize  the Ricci tensor, i.e., $\Ric(e_i)=\lambda_i e_i$. Then we have that 
\begin{align*}
\sum_i \Big \langle B\circ\Ric(e_i),B(e_i) \Big\rangle  
 &=\sum_i\lambda_i \Big \langle B(e_i),B(e_i) \Big\rangle\geq c \sum_i \Big \langle B(e_i),B(e_i) \Big\rangle\\
 &=c\Vert B\Vert^2
\end{align*}
On other hand,   
\begin{align*}
\Big \langle \cR(B),B\Big\rangle
&=
\Big\vert\sum_{ij}\Big\langle [F_{e_j,e_i},B_{e_j}],B_{e_i}\Big\rangle \Big\vert 
=\Big\vert\sum_{ij}\Big\langle F_{e_j,e_i},[B_{e_j},B_{e_i}]\Big\rangle \Big\vert\\
&\leq\sum_{ij} \Big\vert\Big\langle F_{e_j,e_i},[B_{e_j},B_{e_i}]\Big\rangle \Big\vert
\leq  \sum_{ij}  \Vert F_{e_j,e_i}\Vert\Vert [B_{e_j},B_{e_i}]\Vert\\
&\leq\sqrt{2}\sum_{ij}  \Vert F_{e_j,e_i}\Vert\Vert B_{e_j}\Vert \Vert B_{e_i}\Vert 
\end{align*}
Define $A=(a_{ij})$ where $a_{ij}=\Vert F_{e_j,e_i}\Vert$ and $C=(c_{ij})$ where $c_{ij}=\Vert B_{e_i}\Vert\Vert B_{e_j}\Vert$, note that $A$ and $B$ are symmetric and $\Vert C\Vert^2=\Tr(C^2)\leq \Tr(C)^2$. Furthermore note that $\Tr(C)=\Vert B \Vert^2 $ 
\begin{align*}
\Big\langle\cR(B),B \Big\rangle &=\Tr(AC) 
\end{align*}
so 
$
\Big\langle\cR(B),B \Big\rangle 
\leq\sqrt{2}\Vert A\Vert\Vert C\Vert=\sqrt{2}\Vert F_\nabla\Vert\Vert B\Vert^2
$
hence taking the inner product with $B$ in \eqref{eq:Secondvariation} 
\begin{align*}
\Big \langle B\circ\Ric,B\Big\rangle+2\Big \langle \cR(B),B\Big\rangle
&\geq   c\Vert B\Vert^2-2\sqrt{2}\Vert B\Vert^2\Vert F_\nabla \Vert,
\end{align*}
follows that if $\Vert F_\nabla \Vert <\frac{c}{2\sqrt{2}} $ then $\nabla$ is stable as a Yang-Mills connection. We summarize this section in the following 
\begin{proposition}\label{cor:stabilityF}
Let $(M,\cS)$ be a compact,  Sasakian manifold [cf. Definition \ref{thm:app sasakian manifold}] such that the Ricci tensor satisfy $\Ric\geq c I$ for some $c>0$ and $E\to M$ a complex Sasakian vector bundle \cite{Boyer2008}. If $\nabla\in\cA(E)$ is a SDCI with curvature $2$-form $\rF_\nabla$, then $\nabla$ is Yang-Mills connection. Furthermore, if the the curvature $2$--form satisfy $\Vert F_\nabla\Vert< \frac{c}{2\sqrt{2}}$ then  $\nabla$ is stable. 
\end{proposition}



 
 \appendix
\section{Transverse Kähler geometry on Sasakian manifolds}\label{sec:tKahlerGeometry}
In this section we compile some results on the transversal geometry of a Sasakian manifold, for more details we refer the reader to \cite{Boyer2008}. We adopt the following  definition \cite[Theorem~10]{boyer2007sasakian} 
\begin{definition}\label{thm:app sasakian manifold} 
The following are   equivalent definitions of a \textbf{Sasakian manifold} 
$(M,g)$,
\begin{itemize}
\item[(i)] There is a unit Killing  vector field $\xi\in\fX(M)$, such that the section $\Phi\in\Gamma(TM\otimes T^{\ast}M)$, defined by $\Phi(X)=-\nabla_{X}\xi$ satisfies the identity:
\begin{equation}
\label{sasakiana1}
(\nabla_{X}\Phi)(Y)=g(X,Y)\xi-g(\xi,Y)X
\end{equation}
\item[(ii)]  There is a unit killing vector field $\xi\in\fX(M)$, such that the Riemann curvature tensor $R$ of $(M,g)$ satisfies: 
\begin{equation}
\label{sasakiana2}
R(X,\xi)Y= g(\xi,Y)X-g(X,Y)\xi:=-(\nabla_{X}\Phi)(Y)
\end{equation}
\item[(iii)] The metric cone $(\R_{+}\times M,dr^{2}+r^{2}\cdot g)$ is Kähler.
\end{itemize}
\end{definition}

  We introduce the transverse Kähler structure both globally and locally, these two description are equivalent and readers are referred to \cite[Section~2.5]{Boyer2008} for a detailed treatment on this subject.  The action of the Reeb vector   $\xi$ on $\cC(M)$ has local orbits that defines a transverse holomorphic structure on the Reeb foliation $\cF_\xi$ in the following sense. Let $\{U_\alpha:=V_\alpha\times I\}$ be a foliated chart covering $M$, where $V_\alpha\subset \C^n$ and $I\subset\R$ are open sets.  Let  $\pi_\alpha\colon U_\alpha\to V_\alpha\subset \C^n$ be a submersion such that if $U_\alpha\cap U_\beta:=U_{\alpha\beta}\neq\emptyset$
$$
  \pi_\alpha\circ\pi_\beta^{-1}\colon \pi_\beta(U_{\alpha\beta}) \to \pi_\alpha(U_{\alpha\beta})
$$
is biholomorphism.   $\pi_\alpha$ provides a canonical  isomorphism $d\pi_\alpha\colon \rH_p\to \rT_{\pi_\alpha(p)}$ \eqref{eq:split of TM} for every $p\in U_\alpha$ and $\rH=\ker\eta$ is the kernel of $\eta$ generating the following split 
\begin{equation}\label{eq:split of TM}
    TM=\rH\oplus\R\cdot \xi
\end{equation}
Since $\xi $ generates isometries, the restriction  $g\vert_{\rH}$ gives a well defined Hermitian metric $g_\alpha^T$ on $V_\alpha$, this structure is in fact Kähler. Let $(z_1,\cdots,z_n,x)$ be coordinates for $U_\alpha$, the complexification $(\rH\otimes \C)^{0,1}$  is generates by elements of the form 
$ \partial_{z^i}-\eta(\partial_{z^i}\xi),$   where $  \partial_{z^i}:=\frac{\partial}{\partial z^i}$.
  Note that, from  $i_\xi d\eta=0$ and $i_\xi \eta=1$ follows 
$$
  d\eta(\partial_{z^i}-\eta(\partial_{z^i})\xi,\ol{\partial_{z^i}-\eta(\partial_{z^i})\xi})=d\eta(\partial_{z^i}, \ol{\partial}_{z^i}),
$$ 
  i.e., the closed  fundamental $2$-form $\omega_\alpha$ is $d\eta$ restricted  to ${\{x=constant\}\cap U_\alpha}$, hence  $ \pi_\alpha\circ\pi_\beta^{-1}\colon \pi_\beta(U_{\alpha\beta}) \to \pi_\alpha(U_{\alpha\beta})$  gives and isometry of Kähler manifolds. The restriction   $\omega^T_\alpha:=1/2 d\eta\vert_{\{x=constant\}}$ is the fundamental $2$-form for the Hermitian metric $g^T_\alpha$ on each $V_\alpha\subset\C^n$. Hence $\omega^T_\alpha$ is closed and $g^T_\alpha$ is Kähler on $V_\alpha$. The metric   $g^T=\{g^T_\alpha\}$ is defined as a collection of metrics in each coordinate chart with can be identified with the global tensor on $M$ via
\begin{equation}
    \label{eq:transverse metric}
  g^T(X,Y)=\frac{1}{2}d\eta(X,\Phi Y), X, Y\in\Gamma(\rH).
\end{equation} 
Hence, the metric $g^T$ is related to the Sasakian metric $g$ by:  
\begin{equation}
    \label{eq:g g^t relation}
  g = g^T + \eta\otimes\eta.
\end{equation}
Let $p: TM\to \rH$ be the natural  projection,  the connection $\nabla^T$ of  the transverse metric $g^T$ on $\cF_\xi$ is defined by
\begin{equation}\label{eq:nabla T} 
\nabla^T_XY=\begin{cases}
(\nabla_XY )^p,& X,Y\in \Gamma(\rH)  \\[3pt]
[\xi,Y]^p, & X=\xi \qandq Y\in \Gamma(\rH).
\end{cases}
\end{equation}
$\nabla^T $ is  the unique  torsion free  connection  satisfying $\nabla^T g^T = 0,$  also it is shown that  that $\nabla^T J = 0.$ Here we use the identification $d\pi_\alpha\colon \rH_p\to T_{\pi_\alpha(p)}V_\alpha$. Let $TM_\C$, $N_\xi^\C$ and $\rH_\C$ denote the complexification of $TM$,    $N_\xi$ and $\rH$ respectively. \\

A $k$--form  $\alpha\in\Omega^{k}(M)$ is called  \textbf{transverse} if $i_\xi\alpha=0,$ and it is called \textbf{basic form}  if $\alpha$ is transverse and $i_\xi d\alpha=0$. we denote my $\Omega^{k}_{H}(M)=\Gamma(\Lambda^{k}H^{\ast})$ the space of transverse $k$-forms. A transverse form  $\alpha$ is determined  by the values of $\alpha$ in $\Lambda^{p}\rH_\C$. We denote the complexification $\Phi_{\C}\colon  \rH_\C \to   \rH_\C$ of  $\Phi\in\End(TM)$ by the same symbol. 
Since  $\Phi\vert_{\rH}^{2}=-\mathbbm{1}$,   the eigenvalues of  $\Phi$ are $\pm\sqrt{-1}$, (we denote the   imaginary unit by $\ii$) this induces  a decomposition
\begin{equation}\label{eq:Hpq}
\rH_\C=\rH^{1,0}\oplus  \rH^{0,1}
\end{equation}
hence, for $p,q\geq 0$ we set $\rH^{p,q}:= \Lambda^{p} (\rH^{1,0})^\ast\otimes \Lambda^{q} (\rH^{0,1})^\ast \subset\Lambda^{p+q} (\rH)^\ast\otimes_{\R}\C,$ therefore  
\begin{equation}
\label{eq:Quebra1}
\Lambda^{d}(\rH_\C)^\ast=\bigoplus_{i=0}^{d} (\rH^{i,d-i})^\ast \quad\forall\quad 0<d<2n.
\end{equation}
A $p$--form  $\alpha$ is said to be  of type $(k,p-k)$, if the evaluation of $\alpha$ on $ \rH^{i,p-i}$ is zero for all  $i\neq k$, the decomposition \eqref{eq:Quebra1} gives a decomposition into a direct sum of vector bundles 
$ 
\Omega^{k}(\rH_\C)=\bigoplus_{i=0}^{k} \Omega^{i,k-i}(\rH_\C),
$ 
where $\Omega^{p,q}(\rH_\C)$ denotes $\Gamma(\Lambda^p(\rH_\C)^\ast\otimes\Lambda^q(\rH_\C)^\ast)$. Consequently,   we have a decomposition of bundles  
\begin{equation}\label{eq:DecompostionKforms}
\Omega^{k}(M)=\left(\bigoplus_{i=0}^{k} \Omega^{i,k-i}(\rH_\C) \right) \oplus \eta\otimes\left( \bigoplus_{j=0}^{k-1}\Omega^{j,k-j-1}(\rH_\C)\right). 
\end{equation}
\begin{proposition}\cite[Corollary.~3.1]{Biswas2010}\label{prop: omega (1,1)}
The transversal form  $\omega:=d\eta\vert_{H}$ is of type   $(1,1).$
\end{proposition} 
Let $U_\alpha$ and $U_\beta$ be foliated charts such that $U_\alpha\cap U_\beta\neq\emptyset$ with  coordinates on $U_\alpha$ and $U_\beta$ given respectively by  $(z^i,\cdots,z^n,x)$  and $(w^i,\cdots,w^n,y)$, then 
$ \frac{\partial z^i}{\partial\bar{w}^i}=0$ and $\frac{\partial z^i}{\partial  y}=0,$ 
hence, the type  $(p,q)$  of a $p+q$--form $\alpha$  does not depends on the chosen  chart and $\alpha$ can be written as $\alpha=a_{i_1\cdots i_p\bar{j}_1\cdots \bar{j}_q}dz^{i_1}\wedge\cdots\wedge dz^{i_p}\wedge d\zbar^{j_1}\wedge\cdots\wedge d\zbar^{j_q}$, furthermore,  the functions $a_{i_1,\cdots i_p,\bar{j}_1\cdots \bar{j}_q}$ does not depend of  the coordinate $x$.  Thus we have well-defined operators 
\begin{equation}\label{eq:partial_B}
\begin{array}{lll}
\partial_B :\Omega^{p,q}_B\to\Omega^{p+1,q}_B   &\text{and } &
\bar{\partial}_B:\Omega^{p,q}_B\to\Omega^{p,q+1}_B
\end{array}
\end{equation}
The exterior derivative restricts to basic forms $d_B = d\vert_{\Omega^p_B}$, hence we have the split $d_B = \partial_B + \bar{\partial}_B.$  Relatively to the transverse Hodge star operator $\ast_T$, for  \eqref{eq:partial_B} we have adjoins   
\begin{equation}\label{eq:AdjoinPartial_B}
\begin{array}{ll}
  \partial_B^\ast\colon &\Omega^{p,q}_B\to\Omega^{p-1,q}_B \\ 
                        &\partial_B^\ast:= -\ast_T \partial_B^\ast \ast_T   
 \end{array}
  \text{and}\;
 \begin{array}{ll}
   \bar{\partial}_B^\ast \colon &\Omega^{p,q}_B\to\Omega^{p,q-1}_B\\  
                                &\bar{\partial}_B^\ast:=-\ast_T\bar{\partial}_B^\ast \ast_T,
 \end{array}
\end{equation}
We define the  \textbf{transverse curvature operator} to be the curvature of the transverse Levi-Civita connection defined in \eqref{eq:nabla T}, namely 
\begin{equation}
    \label{eq:curvature T} 
\rR^T(X,Y)Z =\nabla^T_X\nabla^T_Y Z-\nabla^T_Y \nabla^T_X Z-\nabla^T_{[X,Y]}Z.
\end{equation} 
  One can   check that $\rR^T(X,\xi)Y = 0$. We have  the following \cite[Lemma~7.3.8]{Boyer2008}   
 \begin{proposition}\label{eq:curvature Sasakian 2}
  On a Sasakian manifold $(M,\eta,\xi,g,\Phi)$ the curvature tensor $\rR$ satisfy:
\begin{multicols}{2}
\raggedcolumns
\begin{enumerate}
\item $\rR(X,\xi)\xi =  X -\eta(X)\xi$
\item $ \rR(X,\xi)Y = g(\xi, Y )X -g(X, Y )\xi$
\item $ \rR(X,Y)\xi =  \eta(Y)X -\eta(X)Y$
\end{enumerate}
\end{multicols}
\end{proposition}
\begin{definition}\label{def:RicciT}
  The \textbf{ transverse Ricci curvature } is defined by the trace of the transverse curvature tensor $\rR^T$ \eqref{eq:curvature T}, i.e., in  an orthonormal basis $\{e_i\}$ of $\rH$ \eqref{eq:split of TM}, the transverse Ricci curvature is defined  by 
$$
  \Ric^T(X,Y)=\sum_i g(\rR^T(X,e_i)e_i,Y)=\sum_i g^T(\rR^T(X,e_i)e_i, Y).
$$
\end{definition}
\begin{proposition}\cite[Proposition~7.3.12]{Boyer2008} 
  \label{prop:Ric^t Ric}
  Let $(M,\xi, \eta, g, \Phi)$ be a K-contact manifold of dimension $2n+1$, then
\begin{itemize}
  \item[(i)] $\Ric(X,\xi)=2n\eta(X)$, for all $X\in TM$
  \item[(ii)] $ \Ric^T(X,Y)=\Ric(X,Y)+2g(X,Y)=\Ric(X,Y)+2g^T(X,Y),$ $\forall$ $X,Y\in \Gamma(\rH)$
\end{itemize}
\end{proposition}  
  The transverse scalar curvature $s^T$ is defined by the trace of  $\Ric^T$. Hence using $i_\xi\eta=1$, $i_\xi d\eta=0$ and \eqref{eq:g g^t relation}, for the \textbf{transverse scalar curvature} we have 
\begin{equation}
\label{eq:scalar curvature}
  s^T=s + 2n
\end{equation} 
A Sasaki-Einstein manifold is a Riemannian manifold $(M, g)$ that is both Sasakian and Einstein, hence
\begin{definition}
\label{def:TkahlerEinstein}
  A Sasakian manifold $(M,\xi, \eta, g, \Phi)$ is called \textbf{transverse Kähler-Einstein} if there exists  a constant $\lambda$ such that $\Ric^T = \lambda g^T$.
\end{definition} 
A Sasaki-Einstein manifold is a Riemannian manifold $(M, g)$ that is both Sasakian and Einstein, hence
\begin{definition}\label{def:sasaki einstein}
A Sasakian manifold $(M,\xi, \eta, g, \Phi)$ is \textbf{Sasakian-Einstein} if $\Ric = 2ng$.
\end{definition}
In particular  Sasakian-Einstein manifolds are   necessary Ricci positive. This can be summarized in the following
\begin{proposition}\cite[Proposition~1.9]{sparks2010sasaki}\label{prop:sasaki eisntein}
Let $(M,\xi, \eta, g, \Phi)$ a Sasakian manifold of dimension $2n+1$, the following are equivalent
\begin{enumerate}
\item $(M,g)$  is Sasaki-Einstein with $\Ric_g=2ng$.
\item The Kähler cone $(\cC(M),\bar{g})$ in Definition \ref{thm:app sasakian manifold} is Ricci-flat, $\Ric_{\bar{g}=0}$.
\item The transverse Kähler structure to the Reeb foliation $\cF_\xi$ is Kähler-Einstein with $\Ric^T=(2n+2)g^T$.
\end{enumerate}
\end{proposition} 

\bibliography{contenido/Bibliografia-2018-06}

\end{document}